\documentclass[english]{compositio}
\usepackage{mathrsfs,mathtools}
\usepackage[english]{babel}
\usepackage[utf8]{inputenc}
\usepackage[T1]{fontenc}
\usepackage[hidelinks]{hyperref}
\usepackage{xcolor}\usepackage[all]{xy}
\theoremstyle{plain}
\newtheorem{theo}{Theorem}[section]
\newtheorem{alphtheo}{Theorem}
\newtheorem{lemm}[theo]{Lemma}
\newtheorem{coro}[theo]{Corollary}
\newtheorem{prop}[theo]{Proposition}
\theoremstyle{definition}

\newtheorem{alphconj}{Conjecture}
\newtheorem{defi}[theo]{Definition}
\theoremstyle{remark}
\newtheorem{rema}[theo]{Remark}
\newtheorem{exem}[theo]{Example}
\def\C{\mathbb{C}}
\def\P{\mathbb{P}}
\def\Q{\mathbb{Q}}
\def\R{\mathbb{R}}
\def\Z{\mathbb{Z}}
\def\O{\mathcal{O}}
\def\D{\mathbb{D}}% unit disk
\def\F{\mathcal{F}}% foliation

\def\bydef{\coloneqq}
\let\leq\leqslant\let\geq\geqslant
\def\ie{\textit{i.e.}} \def\eg{\textit{e.g.}}
\def\ae{\hspace{.5cm}\Vert} % almost everywhere

\makeatletter
\DeclareRobustCommand\diff{\@ifnextchar[{\@@diff}{\@diff}}%]
\def\@diff{\mathop{}\!{\mathrm{d}}}
\def\@@diff[#1]{\cramped{\diff^{#1}}} % argument optionel pour produire diff^{k}
\DeclareRobustCommand\Diff{\@ifnextchar[{\@@Diff}{\@Diff}}%]
\def\@Diff{\mathop{}\!{\mathrm{D}}\cdot}
\def\@@Diff[#1]{\cramped{\mathop{}\!{\mathrm{D}^{#1}}}\cdot} % argument optionel pour produire diff^{k}
\makeatother
\DeclareMathOperator{\logp}{log^{+}} % order
\DeclareMathOperator{\ord}{ord} % order
\DeclareMathOperator{\mult}{mult}

\DeclareMathOperator{\Spec}{Spec}

\DeclareMathOperator{\Aut}{Aut}
\DeclareMathOperator{\ch}{ch}
\DeclareMathOperator{\Grad}{Grad_{\bullet}}
\DeclarePairedDelimiter{\abs}{\lvert}{\rvert} % \abs* adjust size
\DeclarePairedDelimiter{\Abs}{\lVert}{\rVert} % \Abs* adjust size
\DeclarePairedDelimiter{\Set}{\{}{\}}
\def\suchthat{
  \mathchoice
  {\mathrel{\textsl{\Large|}}}
  {\mathrel{\textsl{\large|}}}
  {\mathrel{\textsl{|}}}
  {\mathrel{\textsl{\small|}}}
}
\def\tprod{\mathop{\textstyle\prod}}
\def\tsum{\mathop{\textstyle\sum}}
\def\myfrac#1#2{
  \mathchoice
  {\raisebox{.3ex}{\scalebox{.9}{\(#1\)}}/\raisebox{-.3ex}{\scalebox{.9}{\(#2\)}}}
  {\raisebox{.3ex}{\scalebox{.9}{\(#1\)}}/\raisebox{-.3ex}{\scalebox{.9}{\(#2\)}}}
  {\raisebox{.1ex}{\(\scriptstyle #1\)}/\raisebox{-.1ex}{\(\scriptstyle #2\)}}
  {\raisebox{.1ex}{\(\scriptscriptstyle #1\)}/\raisebox{-.1ex}{\(\scriptscriptstyle #2\)}}
}
\def\T_#1{{\Omega^{\mathrlap{\vee}}}_{#1}}%tangent bundle; better vertical spacing

\begin{document}
%%%
\title{Orbifold hyperbolicity}
\author{Frédéric Campana}
\email{frederic.campana@univ-lorraine.fr}
\address{Institut de Mathématiques Élie Cartan, Université de Lorraine, B.P. 70239, 54506 Vand\oe{}uvre-lès-Nancy Cedex, France \& KIAS, 85 Hoegiro, Dongdaemungu, Seoul 130-722, South Korea.}

\author{Lionel Darondeau}
\email{lionel.darondeau@normalesup.org}
\address{KU Leuven, Departement Wiskunde, Celestijnenlaan 200B, 3001 Heverlee, België.}
\curraddr{IMAG, Univ Montpellier, CNRS, Montpellier, France.}

\author{Erwan Rousseau}
\email{erwan.rousseau@univ-amu.fr}
\address{Institut Universitaire de France \& Aix Marseille Univ, CNRS, Centrale Marseille, I2M, Marseille, France.}

\hypersetup{ pdfauthor={F. Campana \& L. Darondeau \& E. Rousseau}, pdftitle={Orbifold hyperbolicity}}

\thanks{
  This work has been carried out in the framework of Archimède Labex (ANR-11-LABX-0033) and of the AMIDEX project (ANR-11-IDEX-0001-02), funded by the ``Investissements d'Avenir'' French Government program managed by the French National Research Agency (ANR).
  Erwan Rousseau was partially supported by the ANR project ``FOLIAGE'', ANR-16-CE40-0008.
  Lionel Darondeau is a postdoctoral fellow of The Research Foundation - Flanders (FWO).
  L.D and E.R. thank the KIAS where part of this work was done.
}

\begin{abstract}
  We define and study jet bundles in the geometric orbifold category. We show that the usual arguments from the compact and the logarithmic settings do not all extend to this more general framework. This is illustrated by simple examples of orbifold pairs of general type that do not admit any global jet differential, even if some of these examples satisfy the Green--Griffiths--Lang conjecture. This contrasts with an important result of Demailly (2010) proving that compact varieties of general type always admit jet differentials. We illustrate the usefulness of the study of orbifold jets by establishing the hyperbolicity of some orbifold surfaces, that cannot be derived from the current techniques in Nevanlinna theory. We also conjecture that Demailly's theorem should hold for orbifold pairs with smooth boundary divisors under a certain natural multiplicity condition, and provide some evidence towards it.
\end{abstract}

\maketitle

%%%
% END OF FRONT MATTER
%%%

\section{Introduction}
\label{se:intro}
\subsection{Orbifold hyperbolicity}
The main goal of this paper is to define and study the hyperbolicity of orbifold pairs in the spirit of the program developed in~\cite{Ca04}.
A \textsl{smooth orbifold pair} is a pair \((X,\Delta)\), where \(X\) is a smooth projective variety and where \(\Delta\) is a \(\Q\)-divisor on \(X\) with only normal crossings and with coefficients between \(0\) and \(1\).
In analogy with ramification divisors (see below), it is very natural to write
\[
  \Delta
  =
  \sum_{i\in I}
  (1-\myfrac{1}{m_{i}}) \Delta_{i},
\]
with \textsl{multiplicities} \(m_{i}\) in \(\Q_{\geq1}\cup\Set{+\infty}\).
The multiplicity \(1\) corresponds to the so-called ``compact case'' (empty boundary divisor).
The multiplicity \(+\infty\) corresponds to the so-called ``logarithmic case'' (reduced boundary divisor).
The \textsl{canonical bundle} of an orbifold pair \((X,\Delta)\) is the \(\Q\)-line bundle \(K_{X}+\Delta\).

The general philosophy in complex hyperbolicity is that varieties with positive canonical bundles are (weakly) hyperbolic, in the sense that these admit no (or few) nonconstant entire curves.
Here, we consider \textsl{orbifold entire curves} \(f\colon\C\to(X,\Delta)\) \ie{} entire curves \(f\colon\C\to X\) such that \(f(\C)\not\subset\abs{\Delta}\) and \(\mult_{t}(f^{\ast}\Delta_{i})\geq m_{i}\) for all \(i\) and all \(t\in\C\) with \(f(t)\in\Delta_{i}\).
In a modern point of view, these curves are nothing but the morphisms \((\C,\varnothing)\to(X,\Delta)\) in the orbifold category. But these are actually also the central objects of the Nevanlinna theory of values distribution.
These curves have hence been studied extensively since the beginning of the \(20\)\textsuperscript{th} century.

An orbifold pair \((X,\Delta)\) is \textsl{of general type} if its canonical bundle \(K_{X}+\Delta\) is big.
The following natural generalization to the orbifold category of the Green--Griffiths--Lang conjecture will be the common thread of this paper.
\begin{alphconj}
  \label{conj:orbi-lang}
  If \((X, \Delta)\) is an orbifold pair of general type, then there exists a proper closed subvariety \(Z\subsetneq X\) containing the images of all nonconstant orbifold entire curves \(f\colon\C\to(X,\Delta)\).
\end{alphconj}

Since the seminal works of Bloch and Green--Griffiths \cite{GG80}, one successful approach to study hyperbolicity problems in the usual (\ie{} compact or logarithmic) settings is the use of \textsl{jet differentials} vanishing on an ample divisor, which can be viewed as algebraic differential equations satisfied by nonconstant entire curves (see \cite{De95} and \cite{DL01}).
It is most natural to define \textsl{orbifold jet differentials} to be the logarithmic jet differentials acting holomorphically on orbifold entire curves (see Sect.~\ref{se:definition} and Remark~\ref{rema:corr}).

In the direction of the Green--Griffiths--Lang conjecture in the compact case (\(\Delta=\varnothing\)), the jet differential approach culminates with the following remarkable recent theorem of Demailly~\cite{Dem11} (see also~\cite{Mer15} for the case of hypersurfaces in projective spaces):
\begin{theo}[(Demailly)]
  \label{theo:demailly}
  If a variety \(X\) is of general type, it admits nonzero global jet differentials vanishing on an ample divisor.
  (The converse holds too, by~\cite{CP15}.)
\end{theo}
The proof of Demailly can be adapted \textit{mutatis mutandis} to hold in the logarithmic category.
A spontaneous question is hence to extend this result to the broader orbifold setting.

\subsection{Main results}
Using jet differentials, we provide new positive results towards the orbifold Green--Griffiths--Lang conjecture.
The control of the cohomology of orbifold jet differentials tends to be much more difficult than in the usual (\ie{} compact or logarithmic) settings.
Nevertheless, for surfaces, we show that jet differentials can be used to prove hyperbolicity results, in situations where the tools of Nevanlinna theory (\eg{} Cartan's Second Main Theorem) cannot be used in the current state of the art. We combine jet differentials techniques with a generalization to the orbifold setting of results by McQuillan \cite{McQ98} and Bogomolov \cite{Bogo77} on curves tangent to holomorphic foliations on projective surfaces (see Sect.~\ref{sse:foliations}).
As an illustrative example, in Sect.~\ref{se:existence}, we prove the following.
\begin{alphtheo}
  \label{theo:cartan+}
  On \(X\bydef\P^{2}\), let \(\Delta\) consist of \(11\) lines in general position with orbifold multiplicity \(2\), then the orbifold Green--Griffiths--Lang conjecture holds.
  More precisely, any orbifold entire curve \(\C\to(X,\Delta)\) is constant.
\end{alphtheo}
Note that the more negative \(\Omega_{X}\) is and/or the smaller the multiplicities of \(\Delta\) are, the less positive \(K_{X}+\Delta\) is.
Hence, among surfaces, the case of \(\P^{2}\) with multiplicities \(2\) is particularly challenging.
In this paper we will always consider cases where \(K_{X}+\lfloor\Delta\rfloor\) (the reduced part of the canonical divisor) is not already big. If \(K_{X}+\lfloor\Delta\rfloor\) is big, orbifold curves can be dealt with using logarithmic technics.
Therefore varieties with nonpositive cotangent bundles (such as projective spaces, Abelian varieties, K3 surfaces) will be obvious choices for \(X\) to consider in examples.

More generally, we give numerical conditions for which the Riemann--Roch approach yields the existence of orbifold jet differentials vanishing on an ample divisor and we study various interesting geometric settings (see Sect.~\ref{se:existence}).
As an example:
\begin{alphtheo}
  Let \((X,\Delta)\) be a smooth orbifold surface such that \(K_{X}\) is trivial and \(\abs{\Delta}\) is a smooth ample divisor.
  If the orbifold multiplicity is \(m\geq5\) and if \(c_{1}(\abs{\Delta})^{2}\geq 10c_{2}(X)\), then
  \((X,\Delta)\) admits orbifold jet differentials vanishing on an ample divisor.
\end{alphtheo}

Pushing further our investigations, we have realized that the naive analog of Theorem~\ref{theo:demailly} does not hold anymore in the general orbifold setting!
We show that it is actually \emph{necessary} to strengthen the general type assumption in order to get orbifold jet differentials (see Sect.~\ref{se:nonexistence}).
As an illustrative example, we prove that on \(\P^{2}\), if \(\Delta\) is smooth of arbitrary degree, with orbifold multiplicity \(2\), there is no nonzero global jet differential.
More generally, we prove the following.
\begin{alphtheo}
  On \(\P^{n}\), if \(\Delta\) is smooth of arbitrary degree, with orbifold multiplicity \(m\leq n\), there is no nonzero global jet differential.
\end{alphtheo}

Given a pair \((X,\Delta)\) with \(\Delta=\sum(1-\myfrac{1}{m_{i}})\Delta_{i}\), we introduce new natural ``higher order'' orbifold structures on \(X\):
\[
  \Delta^{(k)}
  \bydef
  \sum_{i\in I}
  \left(1-\myfrac{k}{m_{i}}\right)^{+}
  \Delta_{i}
\]
where \(x^{+}\bydef\max\Set{x,0}\).

It is then noteworthy that most known results towards Conjecture~\ref{conj:orbi-lang} coming from Nevanlinna theory can be \textit{a posteriori} reformulated in terms of the positivity of a pair \((X,\Delta^{(\alpha)})\) (see Sect.~\ref{se:hyperbolicity}).
This confirms the naturality of these pairs (and also shows the necessity to work with rational orbifold multiplicities).

To extend Theorem~\ref{theo:demailly}, we propose the following conjecture, for which we can provide some evidence.
\begin{alphconj}
  \label{conj:orbiDem}
  A smooth orbifold \((X,\Delta)\) of dimension \(n\geq2\) with smooth boundary divisor
  admits nonzero global jet differentials vanishing on an ample divisor if and only if
  \((X,\Delta^{(n)})\) is of general type.
\end{alphconj}
The right-to-left implication should hold without the smoothness assumption on the boundary divisor. It holds at least (trivially) for the graded bundle associated to the Green--Griffiths filtration of the bundle of jet differentials (\textit{cf.} Proposition~\ref{prop:filtration} and above it for all notation):
\[
  \mathrm{Gr}_{\bullet}E_{k,N}\Omega_{\pi,\Delta}
  =
  \bigoplus_{\ell\in(\Z_{\geq0})^{k}\colon\Abs{\ell}=N}
  S^{\ell_{1}}\Omega_{\pi,\Delta^{(1)}}
  \otimes
  S^{\ell_{2}}\Omega_{\pi,\Delta^{(2)}}
  \otimes
  \dotsb
  \otimes
  S^{\ell_{k}}\Omega_{\pi,\Delta^{(k)}}.
\]
In the compact setting, the proof of Theorem~\ref{theo:demailly} relies basically on the fact that most sections of this graded bundle actually lifts to sections of the bundle of jet differentials.
This is hence a strong indication that the boundary divisor in our conjecture could fit in this approach.

We prove the left-to-right implication for Abelian varieties (see Sect.\ref{sse:abelian}).
\begin{alphtheo}
  Let \(X\) be an Abelian variety of dimension \(n\geq2\), and let \(\Delta\) be a smooth ample divisor.
  If \((X,\Delta)\) admits nonzero global jet differentials vanishing on an ample divisor then
  \((X,\Delta^{(n)})\) is of general type.
\end{alphtheo}

The striking examples towards Conjecture~\ref{conj:orbiDem} that we provide shed light on the impossibility to solve the Green--Griffiths--Lang conjecture using only jet differentials,
and shows again the relevance of the orbifold framework to test the standard techniques in a broader natural setting.

\subsection{The Core map}
To conclude this introduction, let us explain why orbifold structures arise very naturally in studying hyperbolicity of complex projective manifolds.

Let \(F\colon Y\to X\) be a holomorphic fibration between complex projective manifolds. Let \(\abs{\Delta}\subset X\) be the union of all codimension one irreducible components of the locus over which the scheme-theoretic fibre of \(F\) is not smooth.
For each component \(\Delta_{i}\) of \(\abs{\Delta}\), let \(D_{i}\bydef\sum_{j\in J} m_{i,j} D_{i,j}\) be the union of all components of \(F^{\ast}\Delta_{i}\) that are mapped surjectively onto \(\Delta_{i}\) by \(F\). Then one defines the multiplicity of \(F\) along \(\Delta_{i}\) by \(m_{i}\bydef m(F,\Delta_{i})\bydef\inf\Set{m_{i,j}, j\in J}\) and the \(\Q\)-divisor
\[
  \Delta(F)
  \bydef
  \sum_{i\in I}
  (1-\myfrac{1}{m_{i}}) \Delta_{i}.
\]
The pair \((X,\Delta(F))\) is called the \textsl{orbifold base} of the fibration \(F\). The fibration is said to be \textsl{of general type} if its orbifold base is of general type. A manifold \(Y\) is said to be \textsl{of special type} (or simply \textsl{special}) if there is no fibration of general type \(F\colon Y\to X\) with \(\dim X>0\). Equivalently \(Y\) is special if, for any \(p\geq1\), any rank-one coherent subsheaf \(L\subset\Omega_{Y}^{p}\) has Iitaka dimension \(\kappa(Y,L)<p\).

Then one has the following fundamental structure result:
\begin{theo}[\cite{Ca04}]
  There exists a unique (up to birational equivalence) fibration, called the \textsl{core map}, \(c_{X}\colon X \to C(X)\) such that the general fiber of \(c_{X}\) is special, \(c_{X}\) is constant if \(X\) is special and \(c_{X}\) is a fibration of general type otherwise.
\end{theo}
This construction arises naturally in the study of the birational classification of varieties. Conjecturally, it also describes the behaviour of entire curves (or more generally the Kobayashi metric) in general manifolds \(X\) (without any positivity assumption), as we shall now explain. On the one hand, it is conjectured in~\cite{Ca04} that the Kobayashi pseudometric of a complex projective manifold \(Y\) identically vanishes if and only if \(Y\) is special. On the other hand, we have seen the natural generalization Conjecture~\ref{conj:orbi-lang} of the Green--Griffiths--Lang conjecture.
Assuming this Conjecture~\ref{conj:orbi-lang}, one obtains that (``usual'') entire curves \(\C\to X\) are either contained in the fibers of the core map or in the inverse image by the core map of a proper closed subvariety. In particular, this would prove that if there is a Zariski dense entire curve in \(X\), then \(X\) is special.
In other words, if \(\dim C(X)>0\), then any nonconstant entire curve \(\C\to X\) is algebraically degenerate. The varieties of general type satisfy \(\dim C(X)=\dim(X)\). The varieties of special type satisfy \(\dim C(X)=0\).

Using the above core map theorem, one can also reformulate a famous conjecture of Lang: a smooth projective variety is Brody-hyperbolic (\ie{} does not contain any entire curve) if and only if it does not contain any special subvariety.
The right-to-left implication is an easy corollary of the Green--Griffiths--Lang conjecture.
More generally, without the hyperbolicity assumption, all entire curves \(\C\to X\) should be contained in the union of the special subvarieties of \(X\).

\section{Orbifold hyperbolicity}
\label{se:hyperbolicity}
\subsection{Orbifold entire curves}
Let us consider smooth orbifold pairs \((X,\Delta)\) for which the orbifold divisor \(\Delta\) can be written
\[
  \Delta
  \bydef
  \sum_{i\in I}
  \left(1-\myfrac{1}{m_{i}}\right)
  \Delta_{i},
\]
where \(\sum_{i\in I}\Delta_{i}\) is a normal crossing divisor on \(X\)
and where \(m_{i}\in\Z_{\geq1}\cup\Set{\infty}\) are at first (possibly infinite) integers.
To study hyperbolicity in this setting, one shall define \textsl{orbifold entire curves}. Two definitions could be considered.
\begin{defi}
  \label{defi:curve}
  An \textsl{orbifold entire curve} is a (nonconstant) entire curve \(f\colon\C\to X\)
  such that \(f(\C)\not\subset\abs{\Delta}\)
  and such that for all \(i\in I\) and for all \(t\in\C\) with \(f(t)\in\Delta_{i}\),
  \begin{description}
    \item[\textsl{divisible} orbifold curves]
      the multiplicity \(\mult_{t}(f^{\ast}\Delta_{i})\) at \(t\) is \emph{a multiple} of \(m_{i}\).
    \item[\textsl{geometric} orbifold curves]
      the multiplicity \(\mult_{t}(f^{\ast}\Delta_{i})\) at \(t\) is \emph{at least} \(m_{i}\).
  \end{description}
\end{defi}
The first definition fits well with the category of orbifolds in the stacky sense (or divisible orbifolds) but is usually unsuitable for applications to hyperbolicity questions as we shall now illustrate.

Examples constructed in \cite{Ca05} consist in smooth and simply connected projective surfaces \(S\) admitting a fibration \(g\colon S\to\P^{1}\) of general type.
In the divisible orbifold category, the orbifold base of these fibrations is defined using \(\gcd\) instead of \(\inf\) in the computation of the fibre multiplicities. Although the multiple fibres consist of several components, they are constructed in such way that the ``divisible'' orbifold base is trivial (\ie{} there are no ``divisible'' multiple fibres). Indeed, some components have multiplicity \(2\), while others have multiplicity \(3\) (this would be impossible for elliptic fibrations).

Recall that there is no nonconstant orbifold entire curve \(\C\to C\) (for both definitions) with values in an orbifold curve \((C,\Delta)\) of general type (the orbifold curve is said \textsl{hyperbolic}, see Corollary~\ref{coro:curve-gt} for a proof).
An idea to study the hyperbolicity of the surface \(S\) above is thus to look at the composed maps of the entire curves \(f\colon\C\to S\) with the fibration \(g\).
\begin{itemize}
  \item
    Working in the category of divisible orbifolds, the curves \(g\circ f\colon\C\to\P^{1}\) will certainly be orbifold for the (here trivial) orbifold structure induced by the fibration, but we do not get any restriction on \(f\).

  \item
    However, working in the category of geometric orbifolds, the curves \(g\circ f\colon\C\to\P^{1}\) will be orbifold for the general type orbifold curve \((\P^{1},\Delta(g))\). By hyperbolicity of the base, one obtains the expected algebraic degeneracy of any entire curve \(f\) in the fibers of the fibration \(g\).

    More generally, without assumption on the dimension, one obtains easily algebraic degeneracy (in the fibers of the fibration) for all fibrations of general type on a curve (see \cite{Ca05}).
\end{itemize}

According to these considerations, in all this paper we will consider orbifold curves only in the sense of the second definition. Using this definition, we can also consider rational orbifold multiplicities \(m_{i}\in\Q\).
We will denote \(f\colon\C\to(X,\Delta)\) an entire curve \(f\colon\C\to X\) which is orbifold for the structure \((X,\Delta)\).
As already mentioned in the introduction, these curves are also the curves studied in the well-established Nevanlinna theory of values distribution.

\subsection{Hyperbolicity}
Let us study the question of hyperbolicity of orbifold pairs \((X,\Delta)\). Namely, we want to study the geometry of entire curves \(f\colon\C\to(X,\Delta)\) and obtain some results towards Conjecture~\ref{conj:orbi-lang}.
Almost all known results in this direction come from Nevanlinna theory, more precisely from truncated Second Main Theorems.

\subsubsection{Projective spaces}
The first striking result, due to Cartan~(\cite{Cartan28},\cite[Cor.~3.B.46]{Kob98}), can be reformulated in the following way in our terminology:
\begin{theo}[(Cartan)]
  \label{theo:cartan}
  Let \(H_{1},\dotsc,H_{c}\) be \(c\) hyperplanes in general position in \(\P^{n}\) and consider the orbifold divisor
  \(\Delta\bydef\sum_{i=1}^{c}(1-\myfrac{1}{m_{i}})H_{i}\).
  If \((\P^{n},\Delta^{(n)})\) is of general type,
  then every orbifold entire curve \(f\colon\C\to(\P^{n},\Delta)\) is linearly degenerate.
\end{theo}
Note that the positivity condition
involved in the statement is a strengthening of the assumption of general type. It is typical of the kind of positivity conditions that we will encounter.

Several generalizations of Cartan's theorem have been obtained (see for example \cite{Ru09}) but applications to orbifolds are not so useful because of bad truncation levels.
Very recently a second main theorem with truncation level one has been obtained in \cite{HVX}, which implies the following:
\begin{theo}
  Let \(H\) be a generic hypersurface of degree \(d\geq15(5n+1)n^{n}\). If \(m>d\) then every orbifold entire curve \(f\colon\C\to(\P^{n},(1-\myfrac{1}{m})H)\) is algebraically degenerate.
\end{theo}

We see that in these results one needs either many components or high lower bounds on multiplicities. One of the goal of this work is to develop techniques which will enable to obtain statements on orbifold entire curves without such strong conditions.
Moreover, once algebraic degeneracy of orbifold entire curves is established, it is sometimes possible to look at stronger statements such as \textsl{hyperbolicity}, \ie{} nonexistence of nonconstant orbifold entire curves. This is illustrated by the following result.
\begin{theo}[{\cite[Cor.~4.9]{Rou10}}]
  \label{theo:hyperb}
  Let \(H_{1},\dotsc,H_{c}\) be \(c\) general hypersurfaces of degrees \(d_{i}\) in \(\P^{n}\) and consider the orbifold divisor
  \(\Delta\bydef\sum_{i=1}^{c}(1-\myfrac{1}{m_{i}})H_{i}\).
  If
  \(
  \sum_{i=1}^{c}(1-\myfrac{1}{m_{i}})d_{i} >2n
  \),
  then every orbifold entire curve \(f\colon\C\to(\P^{n},\Delta)\) contained in an algebraic curve is constant.
\end{theo}

Let us return to Theorem~\ref{theo:cartan+}, where we consider \(11\) lines in general position in \(\P^{2}\), with multiplicities \(2\).
In this case, \(K_{\P^{2}}+\Delta^{(2)}=K_{\P^{2}}<0\), so the theorem of Cartan cannot be applied.
However, once one knows algebraic degeneracy of entire curves (this is done in Corollary~\ref{coro:orbiCartan}), Theorem~\ref{theo:hyperb} yields even the hyperbolicity of the pair (\textit{cf.} Corollary~\ref{coro:GGL-P2+}).

\subsubsection{Abelian varieties}
After Cartan, one important result in the same direction is the truncated second main theorem on (semi-)Abelian varieties due to works of Noguchi, Winkelmann and Yamanoi. In particular, one obtains the following confirmation of Conjecture \ref{conj:orbi-lang} (see for example \cite{YamAb}):
\begin{theo}
  \label{theo:orbiab}
  Let \(A\) be an Abelian variety, let \(D\) be a smooth ample divisor and let \(m>1\).
  Then every orbifold entire curve
  \(f\colon\C\to(A,(1-\myfrac{1}{m})D)\)
  is algebraically degenerate.
\end{theo}

\subsubsection{Quotients of bounded symmetric domains}
A last class of examples is given by quotients of bounded symmetric domains.
Let \(D\) be a bounded symmetric domain such that the Bergman metric has holomorphic sectional curvature bounded from above by \(-1/\gamma\), and let \(\Gamma<\Aut(D)\) be a neat arithmetic subgroup. Then \(X\bydef D/\Gamma\) is a smooth quasi-projective algebraic variety and admits a smooth toroidal compactification \(\overline{X}\) with normal crossings boundary \(H\).
In this setting, Aihara and Noguchi have obtained the following result \cite{Nog91}:
\begin{theo}
  If \(K_{\overline{X}}+(1-\myfrac{\gamma}{m})H\) is big, then every orbifold entire curve
  \[
    f\colon\C\to(\overline{X},(1-\myfrac{1}{m})H)
  \]
  is algebraically degenerate.
\end{theo}

\section{Orbifold jet bundles}
\label{se:definition}
Let us now provide more detail on the definition of orbifold jet differentials.
For the logarithmic cotangent bundle we refer to Noguchi~\cite{Nog86} and for the logarithmic jet bundles we refer to Dethloff--Lu~\cite{DL01}.
\subsection{Adapted coverings}
\label{sse:adapted.coverings}
We consider smooth orbifold pairs \((X,\Delta)\).
Such pairs are studied using their \textsl{orbifold cotangent bundles} (\cite{CP15}).
Following the presentation used notably in \cite{C15}, it is natural to define these bundles on certain Galois coverings, the ramification of which is partially supported on \(\Delta\).

An orbifold divisor \(\Delta\) can be written uniquely
\[
  \Delta
  \bydef
  \sum_{i\in I}
  \left(1-\myfrac{1}{m_{i}}\right)
  \Delta_{i},
\]
where \(\sum_{i\in I}\Delta_{i}\) is a normal crossing divisor on \(X\)
and where for each \(i\in I\), \(m_{i}=a_{i}/b_{i}\),
for integers \(a_{i}>b_{i}\geq0\) that are coprime if \(b_{i}>0\).
If \(b_{i}=0\), by convention \(a_{i}=1\).

A Galois covering \(\pi\colon Y\to X\) from a smooth projective (connected) variety \(Y\) will be termed \textsl{adapted} for the pair \((X,\Delta)\) if
\begin{itemize}
  \item
    for any component \(\Delta_{i}\) of \(\abs{\Delta}\), \(\pi^{\ast}\Delta_{i}=p_{i}D_{i}\), where \(p_{i}\) is an integer multiple of \(a_{i}\) and \(D_{i}\) is a simple normal crossing divisor;
  \item
    the support of \(\pi^{\ast}\Delta+\mathrm{Ram}(\pi)\) has only normal crossings, and the support of the branch locus of \(\pi\) has only normal crossings.
\end{itemize}
There always exists such an adapted covering (\cite[Prop. 4.1.12]{Laz}).

Remark that if a covering is adapted for a divisor \(\sum_{i\in I}(1-\myfrac{b_{i}}{a_{i}})\Delta_{i}\), it is adapted for any divisor \(\sum_{i\in I}(1-\myfrac{b'_{i}}{a'_{i}})\Delta_{i}\) with \(a'_{i}\mid a_{i}\).
In particular, one could use a presentation of orbifold pairs with \(a_{i}\) and \(b_{i}\) nonnecessarily relatively prime. In what follows, we will not make this assumption anymore.
It is sometimes also convenient to allow \(a_{i}=b_{i}\).

For \(k\in\mathbb{N}\cup\Set{\infty}\), it will be useful to denote
\[
  \Delta^{(k)}
  \bydef
  \sum_{i\in I}
  \left(1-\myfrac{k}{m_{i}}\right)^{+}
  \Delta_{i},
\]
where \(x^{+}\bydef\max\Set{x,0}\).
As we shall soon illustrate, the ``multiplicities'' \((m_{i}-k)^{+}\in\Z_{\geq0}\cup\Set{\infty}\) appearing in the numerators of \(\Delta^{(k)}\) shall be interpreted geometrically as the minimal multiplicities of the \(k\)th derivative of an orbifold curve along the components \(\Delta_{i}\)
(see Definition~\ref{defi:curve}).
However, the orbifold multiplicity of \(\Delta^{(k)}\) along \(\Delta_{i}\) is \(m_{i}/\min(k,m_{i})\).

By what precedes, if \(\pi\) is an adapted covering for the pair \((X,\Delta)\), it is adapted for all the pairs \((X,\Delta^{(k)})\).
Note that \(\Delta^{(1)}=\Delta\) is the original orbifold divisor,
that \(\Delta^{(0)}=\sum_{i\in I}\Delta_{i}\) contains the support \(\abs{\Delta}=\sum_{i\in I\colon a_{i}>b_{i}}\Delta_{i}\) of \(\Delta\) (round-up),
and that \(\Delta^{(\infty)}=\sum_{i\in I\colon b_{i}=0}\Delta_{i}\) is the logarithmic part of \(\Delta\) (round-down).

Let \(\pi\colon Y\to X\) be a \(\Delta\)-adapted covering. For any point \(y\in Y\), there exists an open neighbourhood \(U\ni y\) invariant under the isotropy group of \(y\) in \(\Aut(\pi)\), equipped with centered coordinates \(w_{i}\) such that \(\pi(U)\) has coordinates \(z_{i}\) centered in \(\pi(y)\) and
\[
  \pi(w_{1},\dotsc,w_{n})
  =
  (z_{1}^{p_{1}},
  \dotsc,
  z_{n}^{p_{n}}),
\]
where \(p_{i}\) is an integer multiple of the coefficient \(a_{i}\) of \((z_{i}=0)\).
Here by convention, if \(z_{i}\) is not involved in the local definition of \(\Delta\) then \(a_{i}=b_{i}=1\).

\subsection{The orbifold cotangent bundle}
If all multiplicities are infinite (\(\Delta=\Delta^{(0)}\)), for any \(\Delta\)-adapted covering \(\pi\colon Y\to X\), we denote
\[
  \Omega_{\pi,\Delta}
  \bydef
  \pi^{\ast}\Omega_{X}(\log \Delta).
\]
Then the argument of \cite[Sect. 2.2]{C15} can be directly adapted to nonstrictly adapted coverings to define the \textsl{orbifold cotangent bundle} to be the vector bundle \(\Omega_{\pi,\Delta}\) fitting in the following short exact sequence:
\begin{equation}
  \label{eq:orbi_cotangent}
  0
  \to
  \Omega_{\pi,\Delta}
  \hookrightarrow
  \Omega_{\pi,\Delta^{(0)}}
  \stackrel{\mathrm{res}}{\longrightarrow}
  \bigoplus_{i\in I\colon m_{i}<\infty}
  \O_{\myfrac{\pi^{\ast}\Delta_{i}}{m_{i}}}
  \to
  0.
\end{equation}
Here the quotient is the composition of the pullback of the residue map
\[
  \pi^{\ast}\mathrm{res}
  \colon
  \pi^{\ast}\Omega_{X}(\log\Delta^{(0)})
  \to
  \bigoplus_{i\in I\colon m_{i}<\infty}
  \O_{\pi^{\ast}\Delta_{i}}
\]
with the quotients
\(
\O_{\pi^{\ast}\Delta_{i}}
\twoheadrightarrow
\O_{\myfrac{\pi^{\ast}\Delta_{i}}{m_{i}}}
\)
(\cite[\textit{loc. cit.}]{C15}).

Alternatively, the sheaf of orbifold differential forms adapted to \(\pi\colon Y\to(X,\Delta)\) is the subsheaf
\(
\Omega_{\pi,\Delta}
\subseteq
\Omega_{\pi,\abs{\Delta}}
\)
locally generated (in coordinates as above) by the elements
\[
  w_{i}^{\myfrac{p_{i}}{m_{i}}}
  \pi^{\ast}(\diff z_{i}/z_{i})
  =
  w_{i}^{-p_{i}(1-\myfrac{1}{m_{i}})}
  \pi^{\ast}(\diff z_{i}).
\]
Accordingly,
\(\Omega_{\pi,\Delta^{(j)}}\)
is the subsheaf locally generated by the elements
\[
  w_{i}^{\min(j,m_{i})p_{i}/m_{i}}
  \pi^{\ast}(\diff z_{i}/z_{i})
  =
  w_{i}^{-p_{i}(1-\myfrac{j}{m_{i}})^{+}}
  \pi^{\ast}(\diff z_{i}).
\]
For any \(j\geq1\), one has the inclusion of sheaves
\[
  \Omega_{\pi,\Delta^{(\infty)}}
  \subseteq
  \Omega_{\pi,\Delta^{(j+1)}}
  \subseteq
  \Omega_{\pi,\Delta^{(j)}}
  \subseteq
  \Omega_{\pi,\abs{\Delta}}
  \subseteq
  \Omega_{\pi,\Delta^{(0)}}.
\]

The \textsl{orbifold tangent bundle} \(\T_{\pi,\Delta}\) is defined to be the dual of \(\Omega_{\pi,\Delta}\),
locally generated by the elements
\[
  w_{i}^{p_{i}(1-\myfrac{1}{m_{i}})}
  \pi^{\ast}(\myfrac{\partial}{\partial z_{i}}).
\]
Clearly, for any \(j\geq1\), one has the inclusion of sheaves
\[
  \T_{\pi,\Delta^{(0)}}
  \subseteq
  \T_{\pi,\abs{\Delta}}
  \subseteq
  \T_{\pi,\Delta^{(j)}}
  \subseteq
  \T_{\pi,\Delta^{(j+1)}}
  \subseteq
  \T_{\pi,\Delta^{(\infty)}}.
\]

\subsection{Orbifold jet differentials}
We will now define orbifold jet differentials of order \(k\), that generalize orbifold symmetric differentials and coincide with these at order \(1\).

In a local trivialization as above, the coordinate system \(z_{i}\) induce jet-coordinates \(\diff[j]z_{i}\) on \(J_{k}X\) corresponding to the Taylor expansion of germs of holomorphic curves \(\C\to X\) (note that many authors use the normalization where jet-coordinates behave as derivatives but it is preferable to rather consider the normalization where these behave as Taylor coefficients).
\begin{defi}
  \label{defi:differentials}
  The sheaf of \textsl{orbifold jet differentials} of order \(k\) is the sheaf of \(\O_{Y}\)-algebras generated in local coordinates as above by the elements
  \[
    w_{i}^{-p_{i}(1-\myfrac{j}{m_{i}})^{+}}
    \pi^{\ast}(\diff[j]z_{i}),
  \]
  for \(1\leq i\leq\dim(X)\) and \(1\leq j\leq k\).
\end{defi}

Note that for a change of (centered) local adapted coordinates \(w\leftrightarrow\tilde{w}\) on \(Y\), for any \(i\) with \(m_{i}>1\), up to reordering of the coordinates, one can assume that \(D_{i}=(w_{i}=0)=(\tilde{w}_{i}=0)\).
Hence there is a fonction \(\varphi_{i}\colon\C^{n}\to\C\) with \(\varphi_{i}(0)\neq 0\) such that \(w_{i}=\tilde{w}_{i}\cdot\varphi_{i}(\pi(\tilde{w}))\) and \(z_{i}=\tilde{z}_{i}\cdot(\varphi_{i}(\tilde{z}))^{p_{i}}\).
One can then check that our definition in local coordinates indeed makes sense, since a simple computation yields
\[
  w_{i}^{-p_{i}(1-\myfrac{k}{m_{i}})^{+}}
  \pi^{\ast}(\diff[k]z_{i})
  =
  \sum_{j=0}^{k}
  \underbrace{
    \frac{\diff[k-j](\varphi_{i}^{p_{i}})\circ\pi(\tilde{w})}
    {(\varphi_{i}\circ\pi(\tilde{w}))^{p_{i}(1-k/m_{i})^{+}}}
  }_{\text{with no pole in }\pi^{\ast}J_{k}X}
  \underbrace{
    \tilde{w}_{i}^{-p_{i}(1-\myfrac{k}{m_{i}})^{+}}
    \pi^{\ast}(\diff[j]\tilde{z}_{i})
  }_{\text{pole order }\leq p_{i}(1-\myfrac{j}{m_{i}})^{+}}.
\]

The sheaf of orbifold jet differentials of order \(k\) is naturally a sheaf of graded algebras whose graded pieces are denoted \(E_{k,N}\Omega_{\pi,\Delta}\),
the sheaf of \textsl{orbifold jet differentials} of \textsl{order} \(k\) and of \textsl{weighted degree} \(N\).
Explicitely, \(E_{k,N}\Omega_{\pi,\Delta}\) is the locally free subsheaf of \(\pi^{\ast}E_{k,N}\Omega_{X}(\log\abs{\Delta})\) generated in local coordinates as above by elements
\[
  \prod_{i=1}^{\dim(X)}
  \left(
    \myfrac{\pi^{\ast}\diff[1]z_{i}}{w_{i}^{p_{i}(1-\myfrac{1}{m_{i}})^{+}}}
  \right)^{\alpha_{i,1}}
  \dotsm
  \left(
    \myfrac{\pi^{\ast}\diff[k]z_{i}}{w_{i}^{p_{i}(1-\myfrac{k}{m_{i}})^{+}}}
  \right)^{\alpha_{i,k}},
\]
such that \(\Abs{\alpha}\bydef\sum_{i,j}j\alpha_{i,j}=N\).
As an example, one has \(E_{1,N}\Omega_{\pi,\Delta}= S^{N}\Omega_{\pi,\Delta}\).

It is clear from Definition~\ref{defi:differentials} that orbifold jet differentials are logarithmic jet differentials
\(\omega\in\pi^{\ast}E_{k,N}\Omega_{X}(\log(\abs{\Delta}))\)
satisfying certain cancelations along \(D\bydef\sum_{i\in I\colon m_{i}>1}\frac{p_{i}}{m_{i}}D_{i}\), as shown by the following rewriting of the former elements
\[
  \prod_{i=1}^{\dim(X)}
  w_{i}^{\frac{p{i}}{m_{i}}(\alpha_{i,1}\min(m_{i},1)+\dotsb+\alpha_{i,k}\min(m_{i},k))}
  \pi^{\ast}(\diff z_{i}/z_{i})^{\alpha_{i,1}}
  \dotsm
  \pi^{\ast}(\diff[k]z_{i}/z_{i})^{\alpha_{i,k}}.
\]

Note that the direct image of the sheaf of \(\Aut(\pi)\)-invariant sections of \(E_{k,N}\Omega_{\pi,\Delta}\)
\[
  E_{k,N}\Omega_{X,\Delta}
  \bydef
  \pi_{\ast}((E_{k,N}\Omega_{\pi,\Delta})^{\Aut(\pi)})
  \subseteq
  E_{k,N}\Omega_{X}(\log\abs{\Delta}),
\]
which is a subsheaf of logarithmic jet differentials, does not depend on the choice of \(\pi\).
Explicitely, \(E_{k,N}\Omega_{X,\Delta}\) is the locally free subsheaf of \(E_{k,N}\Omega_{X}(\log\abs{\Delta})\) generated in local coordinates as above by elements
\[
  \prod_{i=1}^{\dim(X)}
  {z_{i}}^{\left\lceil\myfrac{\left(\alpha_{i,1}\min(m_{i},1)+\dotsb+\alpha_{i,k}\min(m_{i},k)\right)}{m_{i}}\right\rceil}
  (\diff z_{i}/z_{i})^{\alpha_{i,1}}
  \dotsm
  (\diff[k]z_{i}/z_{i})^{\alpha_{i,k}}.
\]

\subsection{Orbifold jet spaces}
Next, we define the jet spaces, which have the crucial property that every orbifold entire curve lifts to the orbifold jet spaces, in a suitable sense.
\begin{defi}
  The orbifold jet space is defined as \(J_{k}(\pi,\Delta)\bydef\Spec\bigoplus_{N} E_{k,N}\Omega_{\pi,\Delta}\).
\end{defi}
In local adapted coordinates \((w_{1},\dotsc,w_{n})\) on \(U\subseteq Y\),
\[
  J_{k}(\pi,\Delta)\rvert_{U}
  =
  U
  \times
  \Spec\Big(\C\big[w_{i}^{-p_{i}(1-\myfrac{j}{m_{i}})^{+}}\pi^{\ast}(\diff[j]z_{i})\big]\Big)
  \cong
  U \times \C^{nk}.
\]

The space \(J_{k}(\pi,\Delta)\) is the total space of a fiber bundle over \(X\), with the natural projection, but for \(k>1\) it is not a vector bundle.
It is a subsheaf of \(\pi^{\ast}J_{k}X\).
For any two integers \(k>\ell\), the restriction of the projection \(\pi^{\ast}J_{k}X\twoheadrightarrow \pi^{\ast}J_{\ell}X\) to \(J_{k}(\pi,\Delta)\) yields a natural surjective map \(J_{k}(\pi,\Delta)\twoheadrightarrow J_{\ell}(\pi,\Delta)\).
For \(k=1\), of course, \(J_{1}(\pi,\Delta)=\T_{\pi,\Delta}\) is the orbifold tangent bundle.

Let \(f\colon(\D,0)\to(X,x)\) be a germ of holomorphic curve and let \(\pi\colon Y\to X\) be an adapted covering for \((X,\Delta)\).
We can construct a Riemann surface \(V\) with a proper surjective holomorphic map \(\rho\colon V\to\D\) such that there is a holomorphic lifting \(g\colon V\to Y\) of \(f\):
\begin{equation}
  \label{eq:star}\tag{\(\star\)}
  \vcenter{\xymatrix{
      V \ar[d]^{\rho} \ar[r]^{g}
  &Y\ar[d]^{\pi}
  \\
  \D \ar[r]^{f}
  &(X,\Delta)
  }}.
\end{equation}

Let \(t\) be a coordinate on \(\D\). Then we can lift the vector field \(\partial/\partial t\) as a meromorphic vector field on \(V\), which we still denote \(\partial/\partial t\).
Then \((\partial/\partial t,\dotsc,\myfrac{1}{k!}\partial^{k}/\partial t^{k})\) is a meromorphic section of \(J_{k}(V)\) and we can consider
\(
(\pi\circ g)_{\ast}(\partial/\partial t,\dotsc,\myfrac{1}{k!}\partial^{k}/\partial t^{k})
\)
to define a meromorphic lifting \(j_{k}^{\star}(g)\colon V\dasharrow J_{k}(Y)\dasharrow\pi^{\ast}J_{k}(X)\).
In a local trivialization of \(\pi^{\ast}J_{k}(X)\) around \(\pi^{-1}(x)\):
\[
  j_{k}^{\star}(g)
  \bydef
  \left(g,f'\circ\rho,\dotsc,f^{(k)}\circ\rho\right).
\]
Hence \(j_{k}^{\star}(g)\colon V\to\pi^{\ast}J_{k}(X)\) is actually holomorphic.

Recall that a holomorphic curve \(f\colon\D\to X\) is termed \textsl{orbifold} for the pair \((X,\Delta)\) if \(f(\D)\not\subseteq\abs{\Delta}\) and if for \(t\in\D\) such that \(f(t)\in\Delta_{i}\), \(\mult_{t}(f^{\ast}\Delta_{i})\geq m_{i}\).
\begin{prop}
  \label{prop:orbifold.curve}
  \(f\colon(\D,0)\to(X,x)\) be a germ of holomorphic curve.
  The following statements are equivalent.
  \begin{enumerate}
    \item\label{def}
      The curve \(f\) is orbifold for the pair \((X,\Delta)\).
    \item\label{noguchi}
      For any (for one) commutative diagram \eqref{eq:star}
      and for any orbifold form \(\omega\in H^{0}(U,\Omega_{\pi,\Delta})\)
      on \(U \supset g(V)\),
      the meromorphic function
      \(
      (g^{\ast}\omega/\rho^{\ast}\diff t)
      \)
      is holomorphic.
    \item\label{jets}
      For any (for one) commutative diagram \eqref{eq:star}, one has for any (for some) \(k\geq 1\)
      \[
        j_{k}^{\star}(g)\in J_{k}(\pi,\Delta).
      \]
  \end{enumerate}
\end{prop}
\begin{proof}
  The problem being local, we can reduce to the following situation
  \[
    \xymatrix{
      u\in\D
      \ar[d]^{\rho}
      \ar[r]^{g}
  &
  \D\ni w
  \ar[d]^{\pi}
  \\
  t\in\D
  \ar[r]^{f}
  &
  \D\ni z
},
  \]
  where
  \(f(t)=t^{\alpha}\varphi(t)\)
  with \(\varphi(0)\neq 0\) ,
  \(g(u)=u^{\beta}\psi(u)\) with
  \(\psi(0)\neq 0\),
  and
  \(\rho(u)=u^{r}\),
  \(\pi(w)=w^{p}\).
  In particular, we have \(\alpha r=\beta p\).

  A section \(\omega\) of \(\Omega_{\pi,\Delta}\) is locally of the form
  \[
    \omega
    =
    w^{-p(1-1/m)}\pi^{\ast}\diff z
    =
    pw^{(p/m)-1}\diff w,
  \]
  with \(1\leq m\leq\infty\).
  One infers that \(g^{\ast}\omega\) vanishes at order \(\beta((p/m)-1)+(\beta-1)=(\alpha/m)r-1\).
  Therefore \(g^{\ast}\omega/\rho^{\ast}\diff t\) is holomorphic if and only if \(\alpha\geq m\).
  This proves the equivalence of \eqref{def} and \eqref{noguchi}.

  Now we prove the equivalence between \eqref{noguchi} and \eqref{jets} for a fixed diagram \eqref{eq:star}.
  Recall that by definition, \(j_{k}^{\star}(g)\) belongs to \(J_{k}(\pi,\Delta)\) if and only if \(\omega(j_{k}^{\star}(g))\) is holomorphic for all jet differentials \(\omega\).
  Such a jet differential being locally of the form
  \[
    \sum
    a_{\alpha}
    \left(
      \myfrac{\pi^{\ast}\diff[1]z}{w^{p(1-1/m)^{+}}}
    \right)^{\alpha_{1}}
    \dotsm
    \left(
      \myfrac{\pi^{\ast}\diff[k]z}{w^{p(1-k/m)^{+}}}
    \right)^{\alpha_{k}},
  \]
  it is necessary and sufficient to check the holomorphicity of
  \[
    \omega_{j}(j_{k}^{\star}(g))
    =
    \left(
      \myfrac{\pi^{\ast}\diff[j]z}{w^{p(1-j/m)^{+}}}
    \right)
    (g,f'\circ\rho,\dotsc,f^{(k)}\circ\rho).
  \]

  When \(m=\infty\), one has a standard logarithmic derivative:
  \[
    \omega_{j}(j_{k}^{\star}(g))
    =
    \frac{f^{(j)}\circ\rho}{g^{p}}
    =
    \frac{f^{(j)}}{f}\circ\rho.
  \]
  It coincides with \(g^{\ast}\omega_{1}/\rho^{\ast}\diff t\) for \(j=1\).
  The vanishing order \(r((\alpha-j)-\alpha)\) is indeed non negative if and only if \(\alpha=\infty\).

  When \(m\) is finite, if \(j\geq m\), there is nothing to check.
  Else, a straightforward computation shows that:
  \[
    \omega_{j}(j_{k}^{\star}(g))
    =
    \frac{f^{(j)}\circ\rho}{g^{p(1-j/m)}}
    =
    \frac{p!}{j!(p-j)!p^{j}}
    \cdot
    \left(
      \frac{(\rho')^{j}f^{(j)}\circ\rho}{(g')^{j} \pi^{(j)}\circ g}
    \right)
    \cdot
    \left(\frac{g^{\ast}\omega_{1}}{\rho^{\ast}\diff t}\right)^{j},
  \]
  (note that \(j\leq m<p\)).
  In particular, for \(j=1\),
  by commutativity of \eqref{eq:star} one has
  \[
    \omega_{1}(j_{k}^{\star}(g))
    =
    g^{\ast}\omega_{1}/\rho^{\ast}\diff t.
  \]
  More generally, since
  \(
  (\rho')^{j}f^{(j)}\circ\rho
  \)
  and
  \(
  (g')^{j} \pi^{(j)}\circ g
  \)
  appear both in the development of the \(j\)th derivative of \(f\circ\rho=\pi\circ g\), these have the same vanishing order.
  Therefore
  \(
  \omega_{j}(j_{k}^{\star}(g))
  \)
  is holomorphic if and only if
  \(
  (g^{\ast}\omega_{1}/\rho^{\ast}\diff t)
  \)
  is holomorphic.
\end{proof}

Note that conversely, any point of \(J_{k}(\pi,\Delta)\) can be obtained as
\(j_{k}^{\star}(g)\) for some diagram \eqref{eq:star}.
Hence, we record the following natural fact, for completeness.
\begin{prop}[(Differentials of orbifold morphisms)]
  Let \(\varphi\colon (X,\Delta)\to(X',\Delta')\) be an \textsl{orbifold morphism} (see~\cite{Ca11}).
  Then for any commutative diagram
  \[
    \xymatrix{
      Y \ar[d]^{\pi} \ar[r]^{\tilde{\varphi}}&
      Y'\ar[d]^{\pi'}\\
      (X,\Delta) \ar[r]^{\varphi}&
      (X',\Delta')
    },
  \]
  where the vertical maps are adapted coverings,
  there is a canonical map
  \(
  \varphi_{\ast}\colon J_{k}(\pi,\Delta)\to J_{k}(\pi',\Delta')
  \),
  coinciding with the (lift of the) \(k\)th differential of \(\varphi\) outside of \(\Delta\).
  At a point corresponding to the \(k\)th jet of a diagram
  \[
    \xymatrix{
      V \ar[d]^{\rho} \ar[r]^{g}&
      Y\ar[d]^{\pi}\\
      \D \ar[r]^{f}&
      (X,\Delta)
    },
  \]
  it is locally given by
  \[
    \left(g,f'\circ\rho,\dotsc,f^{(k)}\circ\rho\right)
    \mapsto
    \left(\tilde{\varphi}\circ g,(\varphi\circ f)'\circ\rho,\dotsc,(\varphi\circ f)^{(k)}\circ\rho\right).
  \]
\end{prop}
\begin{proof}
  The morphism \(\varphi\circ f\colon\D\to(X',\Delta')\) is orbifold.
\end{proof}
\begin{rema}
  \label{rema:pullback}
  This allows one to define the pullback of orbifold jet differentials by orbifold morphisms, in the obvious way.
\end{rema}

Proposition~\ref{prop:orbifold.curve} allows one to evaluate jet differentials on orbifold curves, or on their holomorphic liftings, as follows.
\begin{defi}
  \label{defi:g*P}
  Let \((X,\Delta)\) be a smooth orbifold pair and let \(\pi\colon Y\to X\) be an adapted covering. For a holomorphic lifting \(g\) of an orbifold entire curve as in \eqref{eq:star}, and a global orbifold jet differential \(P\in H^{0}(Y,E_{k,N}\Omega_{\pi,\Delta})\), we denote by \(g^{\ast}P\) the holomorphic function
  \[
    g^{\ast}P
    \bydef
    P(j_{k}^{\star}(g))
    \colon
    V
    \to
    \C.
  \]
\end{defi}
\begin{rema}
  \label{rema:f*P}
  Note that if \(f\colon\C\to(X,\Delta)\) is an orbifold entire curve and if \(P\in H^{0}(X,E_{k,N}\Omega_{X,\Delta})\) is a global orbifold jet differential defined on \(X\), for any diagram \eqref{eq:star}, the function \(g^{\ast}(\pi^{\ast}P)\) is constant in the fibers of \(\rho\).
  We hence get a holomorphic function \(f^{\ast}P\colon\C\to\C\), that moreover does not depend on the diagram \eqref{eq:star}.
  It is of course nothing but \(f^{\ast}P=P(j_{k}(f))\).
\end{rema}
\begin{rema}
  Beware that, as an example, we will from now on denote plainly by \(g^{\ast}\omega\) the function that was until now denoted by \((g^{\ast}\omega/\rho^{\ast}\diff t)\).
\end{rema}
\begin{rema}
  \label{rema:corr}
  Working with rational orbifold multiplicities, it is natural to work in the slightly larger category of \textsl{orbifold correspondences}, defined below, to still define orbifold jet differentials as logarithmic jet differentials acting holomorphically on ``orbifold entire curves''.
\end{rema}
\begin{defi}
  \label{defi:corr}
  An \textsl{orbifold correspondence} \(f\colon \C\rightrightarrows(X,\Delta)\) is an orbifold diagram:
  \[
    \vcenter{\xymatrix{
        V \ar[d]_{\rho} \ar[r]^{g}
  &(X,\Delta')\ar[d]^{\mathrm{id}}
  \\
  (\C,\Delta_{\rho})
  &(X,\Delta)
  }},
\]
such that
\(
g^{\ast}(\Delta'-\Delta)\geq\rho^{\ast}\Delta_{\rho}
\),
where \((\C,\Delta_{\rho})\) is the orbifold base of a covering \(\rho\) of \(\C\).
\end{defi}

\begin{exem}
  Taking \(\Delta'=\Delta\), one recovers the familiar orbifold entire curves.
\end{exem}
\begin{exem}
  Taking \(\Delta=(1-\myfrac{b}{a})\Delta_{1}\subseteq\P^{1}\) and \(\Delta'=(1-\myfrac{1}{a})\Delta_{1}\), one recovers the ``multivalued function'' \(f_{1}\colon t\mapsto t^{\myfrac{a}{b}}\) by taking \(\rho\) the cyclic cover of order \(b\) and \(g_{1}\colon t\mapsto t^{a}\).
  Indeed
  \((\Delta'-\Delta)=\frac{1}{a}(b-1)\) so \(g^{\ast}(\Delta'-\Delta)=(b-1)=\rho^{\ast}\Delta_{\rho}\).
\end{exem}
\begin{exem}
  More generally, for rationals \(\myfrac{a'}{b'}\geq \myfrac{a}{b}\), one recovers the ``multivalued functions'' \(f_{1}\colon t\mapsto t^{\myfrac{a'}{b'}}\) by taking
  \(\Delta'=(1-\myfrac{b}{ab'})\Delta_{1}\), \(\rho\) the cyclic cover of order \(b'\) and \(g_{1}\colon t\mapsto t^{a'}\).
  Indeed
  \((\Delta'-\Delta)=\frac{b}{ab'}(b'-1)\) so \(g^{\ast}(\Delta'-\Delta)\geq(b'-1)=\rho^{\ast}\Delta_{\rho}\).
\end{exem}

Diagram~\eqref{eq:star} and Definition~\ref{defi:corr} being really in the same spirit, all the work of this paper could be extended to the category of orbifold correspondences.
\subsection{Filtration of jet differential bundles}
For each \(q=1,\dotsc,k\), one can define a weighted degree \(\Abs{\cdot}_{q}\) on
\((\Z_{\geq0})^{n\times k}\) by
\[
  \Abs{(\alpha_{i,j})}_{q}
  \bydef
  \sum_{j=1}^{q}
  \sum_{i=1}^{n}
  j\alpha_{i,j}.
\]
For \(q=k\), it corresponds to the usual weighted degree \(\Abs{\cdot}\).
It induces a weighted degree on meromorphic sections of \(\pi^{\ast}E_{k,N}\Omega_{X}\)
using the formula
\[
  \Abs*{
    \tsum_{\Abs{\alpha}=N}
    u_{\alpha}(w)
    \tprod_{i,j}
    \big({\pi^{\ast}z_{i}^{(j)}}\big)^{\alpha_{i,j}}
  }_{q}
  \bydef
  \min(
  \Abs{\alpha}_{q}
  \colon
  u_{\alpha}\not\equiv0
  ).
\]

\begin{prop}[(Green--Griffiths filtration)]
  \label{prop:filtration}
  There is a natural filtration of \(E_{k,N}\Omega_{\pi,\Delta}\) induced by the weighted degrees \(\Abs{\cdot}_{k-1}\dotsc,\Abs{\cdot}_{1}\), with associated graded bundle
  \[
    \mathrm{Gr}_{\bullet}E_{k,N}\Omega_{\pi,\Delta}
    =
    \bigoplus_{\ell\in(\Z_{\geq0})^{k}\colon\Abs{\ell}=N}
    S^{\ell_{1}}\Omega_{\pi,\Delta^{(1)}}
    \otimes
    S^{\ell_{2}}\Omega_{\pi,\Delta^{(2)}}
    \otimes
    \dotsb
    \otimes
    S^{\ell_{k}}\Omega_{\pi,\Delta^{(k)}}.
  \]
\end{prop}
\begin{proof}
  We proceed by induction on the length of tensor products in the summand:
  we will prove that there is a natural filtration of \(E_{k,N}\Omega_{\pi,\Delta}\) induced by the weighted degrees \(\Abs{\cdot}_{k-1}\dotsc,\Abs{\cdot}_{p}\), with associated graded bundle
  \[
    \mathrm{Gr}_{\bullet}^{p}E_{k,N}\Omega_{\pi,\Delta}
    =
    \bigoplus_{p\ell_{p}+\dotsb+k\ell_{k}\leq N}
    E_{p,N-p\ell_{p}-\dotsb-k\ell_{k}}\Omega_{\pi,\Delta}
    \otimes
    S^{\ell_{p}}\Omega_{\pi,\Delta^{(p)}}
    \otimes
    \dotsb
    \otimes
    S^{\ell_{k}}\Omega_{\pi,\Delta^{(k)}}.
  \]
  Since \(E_{1,\ell}\Omega_{\pi,\Delta}=S^{\ell}\Omega_{\pi,\Delta}\), the sought statement corresponds indeed to the case \(p=1\).

  The weighted degree \(\Abs{\cdot}_{k-1}\) induces a descending filtration by subbundles
  \[
    E_{k-1,N}\Omega_{\pi,\Delta}\cong F_{k-1}^{N}\subset\dotsb\subset F_{k-1}^{w+1}\subset F_{k-1}^{w}\subset\dotsb\subset F_{k-1}^{0}=E_{k,N}\Omega_{\pi,\Delta}
  \]
  of \(E_{k,N}\Omega_{\pi,\Delta}\), with
  \[
    F_{k-1}^{w}
    \bydef
    \Set*{
      \omega\in E_{k,N}\Omega_{\pi,\Delta}
      \suchthat
      \Abs{\omega}_{k-1}\geq w
    }.
  \]
  Note that these are indeed subbundles because in a coordinate change, the weighted degree \(\Abs{\cdot}_{k-1}\) can only increase.
  This is an easy corollary of the upper-triangularity of the Faà di Bruno formula.

  We claim that the graded pieces are
  \[
    F_{k-1}^{w}/F_{k-1}^{w+1}
    \cong
    \begin{cases}
      0
    &\text{if }k\nmid N-w,\\
    E_{k-1,w}\Omega_{\pi,\Delta}\otimes S^{\ell_{k}}\Omega_{\pi,\Delta^{(k)}}
    &\text{if }w=N-k\ell_{k},
  \end{cases}
  \]
  and the announced result follows.
  The proof of the claim is standard.
  Let us simply point out that it relies on the simple observation that if one mods out by jet-coordinates of order less than \(k\),
  \(\diff[k](\phi\circ z)=(\phi'\circ z)\cdot\diff[k]z\);
  hence, in the filtration, polynomials in jet-coordinates of order \(k\) behave under coordinates changes \(\phi\) in the exact same way as symmetric differential forms (\ie{} polynomials in jet-coordinates of order \(1\)) do.
  Here the slight subtelty is that we consider orbifold jet differentials: the pole order of a \(k\)th jet coordinate
  \(w_{i}^{-p_{i}(1-k/m_{i})^{+}}\pi^{\ast}\diff[k]z_{i}\)
  for the pair \((X,\Delta)\) is not the same as the pole order of the \(1\)st jet-coordinate
  \(w_{i}^{-p_{i}(1-1/m_{i})}\pi^{\ast}\diff z_{i}\)
  for the pair \((X,\Delta)\) but rather the same as the pole order of the \(1\)st jet-coordinate
  \(w_{i}^{-p_{i}(1-k/m_{i})^{+}}\pi^{\ast}\diff z_{i}\)
  for the pair \((X,\Delta^{(k)})\) (\textit{cf.} Definition~\ref{defi:differentials}).
\end{proof}
\begin{rema}
  Notice that for \(k\gg1\), one has \(\Delta^{(k)}=\Delta^{(\infty)}\), the logarithmic part of \(\Delta\).
\end{rema}

\subsection{Euler characteristic of the Green--Griffiths bundle}
Following Green--Griffiths, we now use the graduation obtained in Proposition~\ref{prop:filtration} and the Riemann--Roch formula to compute the Euler characteristic of \(E_{k,N}\Omega_{\pi,\Delta}\).
\begin{prop}
  \label{prop:RR}
  The Euler characteristic of \(E_{k,N}\Omega_{\pi,\Delta}\) has the asymptotic expansion in \(N\), for fixed \(k\):
  \[
    \chi\bigl(E_{k,N}\Omega_{\pi,\Delta}\bigr)
    =
    \frac{N^{(k+1)n-1}}{(k!)^{n}((k+1)n-1)!}
    \chi_{k}(\pi,\Delta)
    +O\bigl(N^{(k+1)n-2}\bigr),
  \]
  where
  \[
    \chi_{k}(\pi,\Delta)
    \bydef
    (-1)^{n}
    \!\!\sum_{q\in\mathbb{N}^{k}\colon\abs{q}=n}\!\!
    \frac{s_{q_{1}}(\Omega_{\pi,\Delta^{(1)}})}{1^{q_{1}}}
    \dotsm
    \frac{s_{q_{k}}(\Omega_{\pi,\Delta^{(k)}})}{k^{q_{k}}}.
  \]
\end{prop}
\begin{proof}
  We follow in spirit Green and Griffiths~\cite[Prop.~1.10]{GG80}.
  By Proposition~\ref{prop:filtration}
  \[
    \ch E_{k,N}\Omega_{\pi,\Delta}
    =
    \sum_{\Abs{\ell}=N}
    \ch(S^{\ell_{1}}\Omega_{\pi,\Delta^{(1)}})
    \ch(S^{\ell_{2}}\Omega_{\pi,\Delta^{(2)}})
    \dotsm
    \ch(S^{\ell_{k}}\Omega_{\pi,\Delta^{(k)}}).
  \]
  For \(i=1,\dotsc,k\), the orbifold cotangent bundle \(\Omega_{\pi,\Delta^{(i)}}\) is a vector bundle of rank \(n\).
  Let \(\lambda_{1}^{(i)}\), \dots, \(\lambda_{n}^{(i)}\) be a set of Chern roots for it.
  In terms of these Chern roots, we get
  \[
    \ch E_{k,N}\Omega_{\pi,\Delta}
    =
    \sum_{\sum_{i=1}^{k}\sum_{j=1}^{n}ix_{i,j}=N}
    \exp\big(\sum_{i=1}^{k}\sum_{j=1}^{n}x_{i,j}\lambda_{j}^{(i)}\big).
  \]
  Using the sum-integral formula yields
  \[
    \ch E_{k,N}\Omega_{\pi,\Delta}
    =
    N^{kn-1}
    \int_{\sum_{i=1}^{k}\sum_{j=1}^{n}ix_{i,j}=1}
    \exp\big(\sum_{i=1}^{k}\sum_{j=1}^{n}Nx_{i,j}\lambda_{j}^{(i)}\big)
    \diff\omega
    +
    O(N^{kn-2}).
  \]
  Expanding the exponential:
  \[
    \ch E_{k,N}\Omega_{\pi,\Delta}
    =
    N^{(k+1)n-1}
    \int_{\sum_{i=1}^{k}\sum_{j=1}^{n}ix_{i,j}=1}
    \frac{\left(\sum_{i=1}^{k}\sum_{j=1}^{n}x_{i,j}\lambda_{j}^{(i)}\right)^{n}}{n!}
    \diff\omega
    +
    O(N^{(k+1)n-2}).
  \]
  Rescaling:
  \[
    \ch E_{k,N}\Omega_{\pi,\Delta}
    =
    N^{(k+1)n-1}
    \int_{\sum_{i=1}^{k}\sum_{j=1}^{n}x_{i,j}=1}
    \frac{\left(\sum_{i=1}^{k}\sum_{j=1}^{n}x_{i,j}\frac{\lambda_{j}^{(i)}}{i}\right)^{n}}{n!}
    \frac{\diff\omega}{(k!)^{n}}
    +
    O(N^{(k+1)n-2}).
  \]
  Using multinomial formula, one gets
  \begin{multline*}
    \ch E_{k,m}\Omega_{\pi,\Delta}
    =\\
    N^{(k+1)n-1}
    \sum_{\sum q_{i,j}=n}
    \frac{(\lambda_{1}^{(1)})^{q_{1,1}}\dotsm(\lambda_{n}^{(k)})^{q_{k,n}}}
    {1^{\sum q_{1,j}}\dotsm k^{\sum q_{k,j}}}
    \int_{\sum_{i=1}^{k}\sum_{j=1}^{n}x_{i,j}=1}
    \frac{x_{1,1}^{q_{1,1}}}{q_{1,1}!}
    \dotsm
    \frac{x_{k,n}^{q_{k,n}}}{q_{k,n}!}
    \frac{\diff\omega}{(k!)^{n}}
    +
    O(N^{(k+1)n-2}).
  \end{multline*}
  For any \(q\)'s with \(\sum q_{i,j}=n\), by calculus:
  \[
    \int_{\sum_{i=1}^{k}\sum_{j=1}^{n}x_{i,j}=1}
    \frac{x_{1,1}^{q_{1,1}}}{q_{1,1}!}
    \dotsm
    \frac{x_{k,n}^{q_{k,n}}}{q_{k,n}!}
    \diff\omega
    =
    \frac{1}{((k+1)n-1)!}.
  \]
  Factorizing:
  \[
    \ch E_{k,N}\Omega_{\pi,\Delta}
    =
    \frac{N^{(k+1)n-1}}{(k!)^{n}((k+1)n-1)!}
    \sum_{\sum q_{i,j}=n}
    \frac{(\lambda_{1}^{(1)})^{q_{1,1}}\dotsm(\lambda_{n}^{(k)})^{q_{k,n}}}
    {1^{\sum q_{1,j}}\dotsm k^{\sum q_{k,j}}}
    +
    O(N^{(k+1)n-2}).
  \]
  By plain linear algebra manipulations, one then gets
  \[
    \ch E_{k,N}\Omega_{\pi,\Delta}
    =
    \frac{N^{(k+1)n-1}}{(k!)^{n}((k+1)n-1)!}
    \sum_{\sum q_{i}=n}
    \frac{h_{q_{1}}(\lambda^{(1)})\dotsm h_{q_{k}}(\lambda^{(k)})}
    {1^{q_{1}}\dotsm k^{q_{k}}}
    +
    O(N^{(k+1)n-2}),
  \]
  where \(h_{q}\) is the \(q\)th complete symmetric function. It remains to note that a definition of Segre classes is
  \[
    s_{q}(\Omega_{\pi,\Delta^{(i)}})
    =
    (-1)^{q}h_{q}(\lambda^{(i)}).
  \]
  This proves the sought formula for the asymptotic Euler characteristic, by the Riemann--Roch theorem.
\end{proof}

For large jet orders, the asymptotic Euler characteristic is controlled by the canonical bundle of the logarithmic part of \(\Delta\).
\begin{prop}
  \label{prop:chi}
  For an adapted covering \(\pi\colon Y\to(X,\Delta)\) of a smooth orbifold pair,
  \[
    \chi_{k}(\pi,\Delta)
    =
    \frac{\big(K_{X}+\Delta^{(\infty)}\big)^{n}}{n!}
    (\log k)^{n}
    +
    O\bigl((\log k)^{n-1}\bigr).
  \]
\end{prop}
\begin{proof}%pas immédiat en suivant G--G car on a différents fibrés
  We follow again Green--Griffiths~\cite{GG80}, with some slight modifications.
  Recall that one can fix some \(i\) such that \(\Delta^{(p)}\) coincides with \(\Delta^{(\infty)}\) for \(p>i\).
  Then:
  \begin{multline*}
    \sum_{q_{1}+\dotsb+q_{k}=n}
    \frac{
      s_{q_{1}}(\Omega_{\pi,\Delta^{(1)}})
      \dotsm
      s_{q_{k}}(\Omega_{\pi,\Delta^{(k)}})
    }
    {1^{q_{1}}\dotsm k^{q_{k}}}
    =\\
    \sum_{q_{1}+\dotsb+q_{i}+q=n}
    \left(
      \frac{
        s_{q_{1}}(\Omega_{\pi,\Delta^{(1)}})
        \dotsm
        s_{q_{i}}(\Omega_{\pi,\Delta^{(i)}})
      }
      {1^{q_{1}}\dotsm i^{q_{i}}}
      \sum_{q_{i+1}+\dotsb+q_{k}=q}
      \frac{
        s_{q_{i+1}}(\Omega_{\pi,\Delta^{(\infty)}})
        \dotsm
        s_{q_{k}}(\Omega_{\pi,\Delta^{(\infty)}})
      }
      {(i+1)^{q_{i+1}}\dotsm k^{q_{k}}}
    \right).
  \end{multline*}
  Reasoning in the exact same way as in~\cite{GG80}:
  \begin{multline*}
    \sum_{q_{1}+\dotsb+q_{k}=n}
    \frac{
      s_{q_{1}}(\Omega_{\pi,\Delta^{(1)}})
      \dotsm
      s_{q_{k}}(\Omega_{\pi,\Delta^{(k)}})
    }
    {1^{q_{1}}\dotsm k^{q_{k}}}
    =\\
    \sum_{q_{1}+\dotsb+q_{i}+q=n}
    \underbrace{
      \frac{
        s_{q_{1}}(\Omega_{\pi,\Delta^{(1)}})
        \dotsm
        s_{q_{i}}(\Omega_{\pi,\Delta^{(i)}})
      }
      {1^{q_{1}}\dotsm i^{q_{i}}}
    }_{O(1)}
    \left(
      \frac{(\log k)^{q}}{q!}s_{1}(\Omega_{\pi,\Delta^{(\infty)}})^{q}
      +
      O\bigl((\log k)^{q-1}\bigr)
    \right).
  \end{multline*}
  Hence, keeping only the term in \((\log k)^{n}\) (for which \(q_{1}=\dotsb=q_{i}=0\)),
  \[
    (-1)^{n}
    \sum_{q_{1}+\dotsb+q_{k}=n}
    \frac{
      s_{q_{1}}(\Omega_{\pi,\Delta^{(1)}})
      \dotsm
      s_{q_{k}}(\Omega_{\pi,\Delta^{(k)}})
    }
    {1^{q_{1}}\dotsm k^{q_{k}}}
    =
    \frac{(\log k)^{n}}{n!}c_{1}(\Omega_{\pi,\Delta^{(\infty)}})^{n}
    +
    O\bigl((\log k)^{n-1}\bigr).
  \]
  This finishes the proof.
\end{proof}
\begin{rema}
  Note that, in contrast with the compact setting and the logarithmic setting, here the condition \((K_{X}+\Delta^{(\infty)})^{n}>0\) does \emph{not} coincide with the condition of orbifold general type, since \eg{} it reduces to \((K_{X})^{n}>0\) when \(\Delta\neq0\) but \(\Delta^{(\infty)}=0\).
  This tends to show that in order to treat the orbifold Green--Griffiths conjecture one should also deal with higher order cohomology spaces.
\end{rema}

\section{Tautological inequalities and vanishing theorems}
\label{se:vanishing}
\subsection{Nevanlinna Theory and the tautological inequality}
We first recall useful results of Nevanlinna theory, following the point of view of Yamanoi in~\cite{Yam04} (see also the more recent \cite{Ya15} and \cite{PS}). As we shall show below, the orbifold setting fits perfectly with this point of view (\textit{cf}.\ Theorem~\ref{theo:taut}).
Let \(Y\) be a smooth projective manifold. We consider holomorphic curves \(g\colon V\to Y\), where \(V\) is a Riemann surface with a proper surjective holomorphic map \(\rho\colon V\to\C\) (which may be the identity):
\begin{equation}
  \label{eq:yam04}
  \vcenter{\xymatrix{
      V
      \ar[d]^{\rho}
      \ar[r]^{g}
      &
      Y
      \\
      \C
      &
  }}.
\end{equation}
Let \(t\) be the standard complex coordinate on \(\C\) and recall that we denote by \(\myfrac{\partial}{\partial t}\) the meromorphic lifting to \(V\) of the vector field \(\myfrac{\partial}{\partial t}\).

For a real \(r>0\), let \(V(r)\bydef\Set{v\in V\suchthat\abs{\rho(v)}<r}\). Recall the main Nevanlinna functions.
For an effective divisor \(D\bydef(\sigma=0)\) on \(Y\), and a hermitian metric \(\Abs{\cdot}\) on \(\O(D)\),
\begin{itemize}
  \item
    the \textsl{proximity function} to \(D\) of \(g\) is defined as
    \[
      m_{g}(r,D)
      \bydef
      \frac{1}{2\pi\deg\rho}
      \int_{\partial V(r)}
      \logp\frac{1}{\Abs{\sigma\circ g}}
      \cdot
      \rho^{\ast}\diff t;
    \]
  \item
    the \textsl{counting function} of \(D\) is defined as
    \[
      N(r,g^{\ast}D)
      \bydef
      \frac{1}{\deg\rho}
      \int_{1}^{r}
      \left(
        \tsum_{u\in V(s)}\ord_{u}g^{\ast}D
      \right)
      \frac{\diff s}{s};
    \]
  \item
    the \textsl{truncated counting function} of \(D\) is defined as
    \[
      N_{1}(r,g^{\ast}D)
      \bydef
      \frac{1}{\deg\rho}
      \int_{1}^{r}
      \left(
        \tsum_{u\in V(s)}\min\Set{1,\ord_{u}g^{\ast}D}
      \right)
      \frac{\diff s}{s}.
    \]
\end{itemize}
Lastly, for a line bundle \(L\) on \(Y\), the \textsl{height function} of \(g\) with respect to \(L\) is defined as
\[
  T_{g}(r,L)
  \bydef
  \frac{1}{\deg\rho}
  \int_{1}^{r}
  \left(
    \int_{V(s)}
    g^{\ast}c_{1}(L)
  \right)
  \frac{\diff s}{s}
  +O(1).
\]
Recall that the height function enjoys boundedness, additivity and functoriallity properties.

The Nevanlinna functions are related by the following fundamental result.
\begin{theo}[(First Main Theorem)]
  Assume that \(g(V)\not\subset\mathrm{Supp} D\). One has
  \[
    T_{g}(r,\O(D))
    =
    N(r,g^{\ast}D)
    +
    m_{g}(r,D)
    +
    O(1).
  \]
\end{theo}

Let us next recall the classical \textsl{Lemma on logarithmic derivatives}.
\begin{theo}[\cite{Nog85,Yam04}]
  \label{theo:derlog}
  Let \(\xi\) be a meromorphic function on \(V\) considered as a holomorphic function \(V \to \P^{1}\). Then for any \(\ell\geq1\), one has
  \[
    \frac{1}{2\pi\deg\rho}
    \int_{\partial V(r)}
    \logp \abs*{\frac{\frac{\partial^{\ell}}{\partial t^{\ell}}\xi}{\xi}}
    \cdot
    \rho^{\ast}\diff t
    \leq
    O(\logp T_{\xi}(r,[\infty]))
    +
    O(\log r)
    \ae.
  \]
  The symbol \(\Vert\) means that the inequality holds for \(r\geq 0\) outside a set of finite linear measure and \(\logp x= \max\Set{\log x,0}\).
\end{theo}

A geometrical consequence of the Lemma on logarithmic derivatives is McQuillan's ``tautological inequality''.
In the non-orbifold setting: let \(g_{[1]}\) denote the canonical lifting of a nonconstant holomorphic map \(g\colon V\to Y\) to \(\P(\Omega_{Y})\).
From Vojta \cite[Th. 29.6]{Vojta} (see also \cite{PS}), in the classical setting (without boundary):
\begin{theo}[(Tautological Inequality)]
  For an ample line bundle \(A\to Y\), one has:
  \[
    T_{g_{[1]}}(r,\O_{\P(\Omega_{Y})}(1))
    \leq
    N(r,\mathrm{Ram}(\rho))
    +
    O(\logp T_{g}(r,A))
    +
    O(\log r)
    \ae.
  \]
\end{theo}

We will now extend this classical result to the orbifold setting.
Let \((X,\Delta)\) be a smooth orbifold pair and let \(\pi\colon Y\to X\) be a \(\Delta\)-adapted Galois covering.
We consider \textsl{holomorphic liftings} \(g\colon V\to Y\) of orbifold entire curves \(f\colon\C\to(X,\Delta)\), where \(V\) is a Riemann surface with a proper surjective holomorphic map \(\rho\colon V\to\C\). Namely the curves \(f\) and \(g\) fit in the following commutative diagram:
\begin{equation}
  \label{eq:lifting}
  \vcenter{\xymatrix{
      V
      \ar[d]^{\rho}
      \ar[r]^{g}
      &
      Y
      \ar[d]^{\pi}
      \\
      \C
      \ar[r]^{f}
      &
      (X,\Delta)
  }}.
\end{equation}
According to Proposition~\ref{prop:orbifold.curve}, using this diagram, one can then define \(j_{1}^{\star}(g)\colon V\to J_{1}(\pi,\Delta)=\T_{\pi,\Delta}\) and thus \(g_{[1]}\colon V\to\P(\Omega_{\pi,\Delta})\). We fix this notation for later use.
Recall also that \(g^{\ast}P\) refers to the holomorphic function introduced in Definition~\ref{defi:g*P}.

Viewing any jet differential as a polynomial in the orbifold jet coordinates with holomorphic coefficients, one obtains the following important intermediate result.
\begin{coro}[(Lemma on logarithmic derivatives for orbifold jet differentials)]
  \label{coro:LLD-jets}
  Let \(P\in H^{0}(Y,E_{k,m}\Omega_{\pi,\Delta})\) be an orbifold jet differential.
  Let \(A\to X\) be an ample line bundle.
  If \(g^{\ast}P\not\equiv0\), then one has:
  \[
    \frac{1}{2\pi\deg\rho}
    \int_{\partial V(r)}
    \logp \abs{g^{\ast}P}
    \cdot
    \rho^{\ast}\diff t
    \leq
    O(\logp T_{g}(r,\pi^{\ast}A))
    +
    O(\log r)
    \ae.
  \]
\end{coro}
\begin{proof}
  We refer to the proof of Theorem A7.5.4 in \cite{Ru01}, which can easily be adapted.
  In order to use Theorem~\ref{theo:derlog}, remind that the orbifold jet coordinates of \(g\) are obtained by applying \({\partial^{\ell}}/{\partial t^{\ell}}\) to \(\pi\circ g\) coordinatewise.
\end{proof}

A key feature of the orbifold tautological inequality is that, using the orbifold cotangent bundle instead of the usual cotangent bundle, one is able to get rid of the ramification term \(N(r,\mathrm{Ram}(\rho))\) for the maps \(g\) stemming from orbifold entire curves:
\begin{theo}[(Orbifold Tautological Inequality)]
  \label{theo:taut}
  Let \(g\colon V\to Y\) be the holomorphic lifting of an orbifold entire curve \(f\colon\C\to(X,\Delta)\).
  For an ample line bundle \(A\to X\), one has:
  \[
    T_{g_{[1]}}(r,\O_{\P(\Omega_{\pi,\Delta})}(1))
    \leq
    O(\logp T_{g}(r,\pi^{\ast}A))
    +
    O(\log r)
    \ae.
  \]
\end{theo}
\begin{proof}
  We follow the approach used by Vojta~\cite{Vojta}, to which we refer for the geometric interpretation of the proof.
  The rough idea is to see the integral in the Lemma on logarithmic derivatives for jet differentials as a proximity function to infinity, in an appropriate compactification.
  Let \(S\) be the total space of \(\T_{\pi,\Delta}\) and let \(\overline{S}=\P(\Omega_{\pi,\Delta}\oplus \O_{Y})\).
  Let \([\infty]\) denote the divisor \(\overline{S}\setminus S\).
  Let \(p\colon P\to\overline{S}\) be the blow-up of \(\overline{S}\) along the image \([0]\) of the zero section, let \(E\) denote its exceptional divisor and let \(q\colon P\to\P(\Omega_{\pi,\Delta})\).
  There is a lifting \(g_{[1]}^{\diamond}\) of \(g\) in \(\P(\Omega_{\pi,\Delta}\oplus\O_{Y})\) and a lifting \(\phi\) to \(P\).
  To sum up, one has the commutative diagram:
  \[
    \xymatrix{ &&&&P\ar_{q}[dl]\ar^{p}[dr]& \\ V\ar^{\rho}[d]\ar_{g_{[1]}}@/_{.8pc}/[rrr]\ar_{g}@/_{2pc}/[rrrr]\ar_{g_{[1]}^{\diamond}}@/_{3.2pc}/[rrrrr]\ar^{\phi}@/^/[urrrr] &&&\P(\Omega_{\pi,\Delta})\ar[r]&Y&\P(\Omega_{\pi,\Delta}\oplus\O_{Y})\ar[l] \\ \C }
  \]

  One has then (cf.~\cite{Vojta} for more details):
  \[
    p^{\ast}\O_{\P(\Omega_{\pi,\Delta}\oplus\O_{Y})}(1)
    \cong
    q^{\ast}\O_{\P(\Omega_{\pi,\Delta})}(1)\otimes\O(E)
    \cong
    q^{\ast}\O([\infty])\otimes\O(E).
  \]
  Hence:
  \[
    T_{g_{[1]}}(r,\O_{\P(\Omega_{\pi,\Delta})}(1))
    =
    T_{\phi}(r,q^{\ast}\O_{\P(\Omega_{\pi,\Delta})}(1))+O(1)
    \\
    =
    T_{g_{[1]}^{\diamond}}(r,[\infty])-T_{\phi}(r,E)+O(1).
  \]
  Now, since \(g\) is nonconstant, \(\phi(V)\not\subset E\), and \(T_{\phi}(r,E)\) is bounded from below.
  It remains to control \(T_{g_{[1]}^{\diamond}}(r,[\infty])\), using the First Main Theorem.
  By the Lemma on logarithmic derivatives, \(m_{g_{[1]}^{\diamond}}(r,[\infty])\) is bounded from above by \(O(\logp T_{g}(r,\pi^{\ast}A))+O(\log r)\).
  Lastly, since \(g\) is the holomorphic lifting of an orbifold curve, the map \(g_{[1]}^{\diamond}\) is holomorphic (cf.\ Prop.~\ref{prop:orbifold.curve}), and therefore \(N_{g_{[1]}^{\diamond}}(r,[\infty])=0\).
  This ends the proof.
\end{proof}

As an immediate corollary, one recovers the hyperbolicity of orbifold curves of general type.
\begin{coro}
  \label{coro:curve-gt}
  Let \((X, \Delta)\) be a smooth orbifold curve and let \(A\to X\) be an ample line bundle.
  For any orbifold entire curve \(f\colon\C\to(X,\Delta)\),
  one has:
  \[
    T_{f}(r,K_{X}+\Delta)
    \leq
    O(\logp T_{f}(r,A))
    +
    O(\log r)
    \ae.
  \]

  In particular, if \(K_{X}+\Delta=A>0\) then there is no entire curve \(f\colon\C\to(X,\Delta)\).
\end{coro}
\begin{proof}
  For curves, the projection
  \(p\colon\P(\Omega_{\pi,\Delta})\to Y\) is an isomorphism and
  \(\O(1)\cong p^{*}\Omega_{\pi,\Delta}\cong p^{\ast}\pi^{\ast}(K_{X}+\Delta)\).
  Therefore by Theorem~\ref{theo:taut}, one has:
  \[
    T_{f}(r,K_{X}+\Delta)
    \leq
    O(\logp T_{f}(r,A))
    +
    O(\log r)
    \ae.
  \]
  Therefore \(f\) extends to an orbifold morphism \(\bar{f}\colon(\P^{1},D)\to(X,\Delta)\) where \(D\) is necessarily supported at infinity. \(\deg (K_{\P^{1}}+D)<0\) and thus \(\bar{f}\) has to be constant by the Riemann--Hurwitz formula.
\end{proof}

\subsection{A vanishing theorem for orbifold jet differentials}
Another immediate application of the tautological inequality is the following vanishing theorem for orbifold symmetric differentials vanishing on an ample divisor.
\begin{coro}
  Let \((X,\Delta)\) be a smooth orbifold pair,
  and let \(\pi\colon Y\to X\) be an adapted covering.
  If \(P\in H^{0}(Y,S^{\ell}\Omega_{\pi,\Delta}\otimes\pi^{\ast}A^{\vee})\) is a global orbifold symmetric differential vanishing on an ample divisor \(A\to X\), then for any holomorphic lifting \(g\colon V\to Y\) of an orbifold entire curve \(f\colon\C\to(X,\Delta)\), one has \(g^{\ast}P\equiv0\).
\end{coro}
\begin{proof}
  Considering the projectivization \(p\colon\P(\Omega_{\pi,\Delta})\to Y\), the symmetric differential \(P\) can be seen as a global section
  \(\tilde{P}\in H^{0}(\P(\Omega_{\pi,\Delta}),\mathcal{L})\),
  where
  \(\mathcal{L}\bydef\O_{\P(\Omega_{\pi,\Delta})}(\ell)\otimes p^{\ast}(\pi^{\ast}A^{\vee})\).
  Let \(g_{[1]}\) be the lifting to \(\P(\Omega_{\pi,\Delta})\) of \(g\) (such that \(g=p\circ g_{[1]}\)).
  Should \(g^{\ast}P=\tilde{P}(g_{[1]})\) not vanish, then, by the boundedness, additivity and functoriallity properties of the height function, one would get that
  \[
    T_{g_{[1]}}(r,\mathcal{L})
    =
    \ell\cdot T_{g_{[1]}}(r,\O(1))
    -
    T_{g}(r,\pi^{\ast}A)
  \]
  is bounded from below. Theorem~\ref{theo:taut} then implies that \(T_{f}(r,A)=T_{g}(r,\pi^{\ast}A)=O(\log r)\).
  Therefore \(f\) extends to an orbifold morphism \(\bar{f}\colon(\P^{1},D)\to(X,\Delta)\) where \(D\) is necessarily supported at infinity. Since \(\deg (K_{\P^{1}}+D)<0\), \(g^{\ast}P\) has to vanish, a contradiction.
\end{proof}
We shall now extend this result to higher order jet differentials.
Let us first settle the case of orbifold curves, in which the existence of orbifold jet differentials gives us an even stronger conclusion.
\begin{lemm}
  If an orbifold pair \((X,\Delta)\) is not of general type,
  then
  \[
    H^{0}(Y,E_{k,N}\Omega_{\pi,\Delta}\otimes\pi^{\ast}A^{\vee})
    =
    \Set{0},
  \]
  for any adapted covering \(\pi\colon Y\to X\),
  for all  \(k\geq1\) and \(N\geq1\),
  for any ample line bundle \(A\to X\).
\end{lemm}
\begin{proof}
  Recall the graduation obtained from the Green--Griffiths filtration:
  \[
    \Grad E_{k,N}\Omega_{\pi,\Delta}
    \otimes
    \pi^{\ast}A^{\vee}
    =
    \bigoplus_{\Abs{\ell}=N}
    S^{\ell_{1}}\Omega_{\pi,\Delta}
    \otimes
    S^{\ell_{2}}\Omega_{\pi,\Delta^{(2)}}
    \otimes
    \dotsb
    \otimes
    S^{\ell_{k}}\Omega_{\pi,\Delta^{(k)}}
    \otimes
    \pi^{\ast}A^{\vee},
  \]
  and remark that for \(j=1,\dotsc,k\), one has
  \(
  S^{\ell_{j}}\Omega_{\pi,\Delta^{(j)}}
  \subseteq
  S^{\ell_{j}} \Omega_{\pi,\Delta}
  \subseteq
  (\Omega_{\pi,\Delta})^{\otimes\ell_{j}}
  \).
  Recall also from \cite{CP15} that if for some integer \(q>0\) and some ample line bundle \(A\), the vector bundle
  \((\Omega_{\pi,\Delta})^{\otimes q}\otimes \pi^{\ast}A^{\vee}\) has a nonzero global section,
  then the pair \((X,\Delta)\) is of general type.
  One infers that under the assumption of the Lemma, for any \(\ell\),
  the graded bundle
  \(
  \Grad E_{k,N}\Omega_{\pi,\Delta}\otimes\pi^{\ast}A^{\vee}
  \)
  has no global section.
  This fact holds \textit{a fortiori} for the bundle
  \(
  E_{k,N}\Omega_{\pi,\Delta}\otimes\pi^{\ast}A^{\vee}
  \)
  itself.
\end{proof}

\begin{coro}
  \label{coro:curve}
  If an orbifold curve \((X,\Delta)\) admits a nonconstant orbifold entire curve \(f\colon\C\to(X,\Delta)\),
  then
  \[
    H^{0}(Y,E_{k,N}\Omega_{\pi,\Delta}\otimes\pi^{\ast}A^{\vee})
    =
    \Set{0},
  \]
  for any adapted covering \(\pi\colon Y\to X\),
  for all  \(k\geq1\) and \(N\geq1\),
  for any ample line bundle \(A\to X\).
\end{coro}
\begin{proof}
  From \cite{CW09} we have that \((X, \Delta)\) contains an orbifold entire curve \(f\colon\C\to(X,\Delta)\) if and only if \(\deg(K_{X}+\Delta)\leq0\). Therefore, if an orbifold curve \((X,\Delta)\) admits a nonconstant orbifold entire curve \(f\colon\C\to(X,\Delta)\) then \((X,\Delta)\) is not of general type and the previous lemma gives 
  \(H^{0}(Y,E_{k,N}\Omega_{\pi,\Delta}\otimes\pi^{\ast}A^{\vee}) = \Set{0}\).
\end{proof}

Now, we can extend the fundamental vanishing theorem of the jet differentials theory to the orbifold setting.
\begin{theo}
  \label{theo:OFVT}
  Let \((X,\Delta)\) be a smooth orbifold pair,
  and let \(\pi\colon Y\to X\) be an adapted covering.
  If \(P\in H^{0}(Y,E_{k,N}\Omega_{\pi,\Delta}\otimes\pi^{\ast}A^{\vee})\) is a global orbifold jet differential vanishing on an ample divisor \(A\to X\),
  then for any holomorphic lifting \(g\colon V\to Y\) of an orbifold entire curve,
  one has \(g^{\ast}P\equiv0\).
\end{theo}
\begin{proof}
  We follow the classical proof (see for example Theorem A7.5.5 in \cite{Ru01}).
  Let us show that \(f\) extends to a rational curve.
  Then, one gets an orbifold morphism \(\bar{f}\colon(\P^{1},D)\to(X,\Delta)\), where \(D\) is necessarily supported at infinity, together with a holomorphic lifting \(\bar{g}\). According to Remark~\ref{rema:pullback}, the jet differential \(P\) then pullbacks to a jet differential on \((\P^{1},D)\). Now, by construction, \((\P^{1},D)\) admits an (orbifold) entire curve. By Corollary~\ref{coro:curve}, it follows that the pullback of \(P\) (and therefore \(g^{\ast}P\)) vanishes identically.

  To show that \(f\) extends to a rational curve, by a classical result, it suffices to establish that
  \(T_{f}(r,A) = O(\log r)\), or equivalently that
  \(T_{g}(r,\pi^{\ast}A)=O(\log r)\).

  Since \(P\) vanishes on \(A\), viewing \(g^{\ast}P\) as a holomorphic function \(V\to\P^{1}\), one has
  \[
    T_{g}(r,\pi^{\ast}A)
    \leq
    O(T_{g^{\ast}P}(r,[\infty])).
  \]

  Now, recall from Definition~\ref{defi:g*P} that the function \(g^{\ast}P\colon V\to\C\) is holomorphic.
  Hence, one has \(N_{g^{\ast}P,[\infty]}\equiv0\).
  Furthermore, applying Corollary~\ref{coro:LLD-jets}, one obtains that the proximity function to infinity of \(g^{\ast}P\) satisfies:
  \[
    m_{g^{\ast}P}(r,[\infty])
    =
    O(\logp T_{g}(r,\pi^{\ast}A))
    +
    O(\log r)
    \ae.
  \]
  Therefore, one has
  \[
    T_{g}(r,\pi^{\ast}A)
    \leq
    O(\logp T_{g}(r,\pi^{\ast}A))
    +
    O(\log r).
  \]
  It follows that
  \(T_{g}(r,\pi^{\ast}A) = O(\log r)\),
  which ends the proof.
\end{proof}

A second version of the vanishing theorem, expressed directly on \(X\), is the following.
\begin{coro}
  \label{coro:OFVT2}
  If \(P\in H^{0}(X,E_{k,N}\Omega_{X,\Delta}\otimes A^{\vee})\) is a global orbifold jet differential vanishing on an ample divisor \(A\to X\), then for any orbifold entire curve \(f\colon\C\to(X,\Delta)\), one has \(f^{\ast}P\equiv0\).
\end{coro}
\begin{proof}
  It follows at once from Remark~\ref{rema:f*P} and from Theorem~\ref{theo:OFVT}.
\end{proof}

\subsection{Orbifold curves tangent to holomorphic foliations}
\label{sse:foliations}
In this section, we will extend to the orbifold setting McQuillan's degeneracy results for entire curves tangent to foliations on surfaces of general type \cite{McQ98} (see also \cite{EG} for the logarithmic setting and \cite{PS} for related results in the setting of parabolic Riemann surfaces).

\begin{theo}
  \label{theo:fol}
  Let \((X,\Delta)\) be a smooth orbifold surface of general type with a holomorphic foliation \(\F\). Any orbifold entire curve tangent to \(\F\) is algebraically degenerate.
\end{theo}

Let \(D\bydef\lceil \Delta \rceil\) and \(f_{[1]}\colon\C\to\P(\Omega_{X}(\log D))\) be the lifting of \(f\). We shall use the following tautological inequality due to McQuillan (see \cite{Vojta}):
\begin{equation}
  \label{eq:logtaut}
  T_{f_{[1]}}(r,\O(1))
  \leq
  N_{1}(r,f^{\ast}D)
  +
  O(\logp T_{f}(r,A))
  +
  O(\log r)
  \ae,
\end{equation}
where \(A\) is an ample line bundle on \(X\).

Let us recall the construction of Ahlfors currents associated to entire curves.
Let \(\eta \in A^{2}(X)\) be a \(2\)-form. Let \(T_{r}(\eta)\bydef\frac{T_{f,\eta}(r)}{T_{f,\omega}(r)}\). This defines a family of positive currents of bounded mass from which one can extract a closed postive current \(T\bydef\lim_{r_{n}} T_{r_{n}}\).

\begin{proof}
  We suppose that \(f\colon\C\to(X, \Delta)\) is a Zariski-dense orbifold curve. Let us prove that
  \[
    T(K_{X}+\Delta) \leq 0,
  \]
  thus contradicting that \((X, \Delta)\) is of general type.

  Let \(S \subset \P(\Omega_{X}(\log D))\) be the surface induced by the foliation \(\F\) and let \(\pi\colon S\to X\) be the projection.
  \(S\) contains \(f_{[1]}(\C)\) and , supposing that \(S\) dominates \(X\), \(S\) is equipped with a foliation \(\F_{0}\).
  After some blow ups, we obtain a foliated smooth surface \((S_{m}, D_{m}, \F_{m}) \to (S, \pi^{-1}(D), \F_{0})\), \ie{} \(S_{m}\) is smooth,  \(D_{m}\) is normal crossing and \(\F_{m}\) has reduced singularities. Let \(D_{m}=C+B\) where \(C\) is the invariant part of \(D_{m}\) by \(\F_{m}\). We have an exact sequence
  \[
    0 \to \mathcal{N}^{*}(C) \to T^{*}_{S_{m}}(\log D_{m}) \to K_{\F_{m}}(B).\mathcal{I}_{Z} \to 0,
  \]
  where \(\mathcal{I}_{Z}\) is an ideal supported on the singularity set \(Z\) of \(\F_{m}\).

  Now, we apply the logarithmic tautological inequality \eqref{eq:logtaut} which gives
  \[
    T_{{f_{m}}_{[1]}}(r,L)
    \leq
    N_{1}(r,f_{m}^{\ast}D_{m})
    +
    O(\logp T_{f}(r,A)+ \log r)
    \ae,
  \]
  where \(L=\mathcal{O}_{\P(\Omega_{\tilde{S}}(\log D_{m}))}(1)\), \(f_{m}\) and \({f_{m}}_{[1]}\) are the lifts of \(f\).

  We have
  \[
    L_{|Y}=p^{*}K_{{\F_{m}}}(B) \otimes\mathcal{O}(-E_{m}),
  \]
  where \(L_{|Y}\) denotes the restriction of \(L\) to the graph \(Y\) of the foliation, \(p\colon Y\to S_{m}\) the projection and \(E_{m}\) is the exceptional divisor.

  Therefore we obtain
  \[
    T_{{f}, K_{X}+D}(r)
    \leq
    T_{{f_{m}}, K_{S_{m}}+D_{m}}(r);
  \]
  hence
  \[
    T_{{f}, K_{X}+D}(r)
    \leq
    N_{1}(r,f^{\ast}D)
    +
    T_{{f_{m}}_{[1]}}(r,E_{m})
    +
    T_{f_{m}}(r,\mathcal{N}^{*}(C))
    +
    O(\logp T_{f}(r,A)+ \log r)
    \ae.
  \]

  Since \(f\) is an orbifold curve, we have
  \[
    m_{i}N_{1}(r,f^{\ast}\Delta_{i})
    \leq
    N(r,f^{\ast}\Delta_{i})
    \leq
    T_{f}(r,\Delta_{i}).
  \]

  This gives
  \[
    T(K_{X}+\Delta)
    \leq
    T'_{m}(E_{m})
    +
    T_{m}(\mathcal{N}^{*}(C)),
  \]
  where \(T'_{m}\) is the current associated to \({f_{m}}_{[1]}\).

  To finish the proof, we shall now use the two following results of Brunella \cite{Bru99} (and McQuillan \cite{McQ98}): \(T_{m}(\mathcal{N}^{*}(C)) \leq 0\) and \( T_{m}(E_{m}) \to 0\) as \(m \to \infty\) \ie{} performing infinitely many blow ups.
\end{proof}

Let us say that a holomorphic foliation \(\F\) on \(X\) is a \(\Delta\)-foliation if \(\pi^{\star}\F\) is a subsheaf of the orbifold tangent bundle \(T_{\pi,\Delta}\bydef\T_{\pi,\Delta}\).

\begin{theo}
  \label{theo:fol2}
  Let \((X, \Delta)\) be a smooth orbifold surface of general type with a \(\Delta\)-holomorphic foliation \(\F\) with reduced singularities, then any (\emph{orbifold or not}) entire curve tangent to \(\F\) is algebraically degenerate.
\end{theo}
\begin{proof}
  We suppose that \(f\colon\C\to X\) is a Zariski-dense curve tangent to \(\F\).
  We have the exact sequence \(0 \to \F \to T_{X} \to \mathcal{N}\).
  We have \(T(K_{\F}) \leq 0\) by a result of McQuillan (see \cite{Bru99}).
  We also have \(T(N^{*}(\Delta)) \leq T(N^{*}(\lceil \Delta \rceil))\leq 0\) by the already mentioned result of Brunella.
  Therefore we obtain, \(T(K_{X}+\Delta)=T(K_\F+N^{*}(\Delta)) \leq 0\),
  giving a contradiction.
\end{proof}

\begin{coro}
  Let \((X, \Delta)\) be a canonical orbifold surface of general type (\ie{} the pair \((X, \Delta)\) has canonical singularities). If \(\F\) is a \(\Delta\)-holomorphic foliation then any entire curve tangent to \(\F\) is algebraically degenerate.
\end{coro}
\begin{proof}
  By Seidenberg's theorem we can do some blow ups such that on \(\tilde{X}\) the induced foliation \(\tilde{\F}\)
  has only reduced singularities. Let us denote \(\tilde{\Delta}\) the strict transform of \(\Delta\). Then \((\tilde{X}, \tilde{\Delta})\) is a smooth orbifold of general type thanks to the hypothesis that \((X, \Delta)\) is canonical. Therefore we can apply Theorem~\ref{theo:fol2} to conclude.
\end{proof}

\section{Existence of orbifold jet differentials on varieties of general type}
\label{se:existence}
\subsection{Order-one jet differentials}
An immediate application of Theorem~\ref{theo:fol} is the following result (see also \cite{Rou10}).

\begin{theo}
  \label{theo:GGO1}
  Let  \((X,\Delta)\) be a smooth orbifold surface of general type.
  If one has
  \[
    H^{0}\big(X,{\textstyle\bigoplus_{N\geq1}}S^{N}\Omega_{X,\Delta}\otimes L^{\vee}\big)
    \neq
    \Set{0},
  \]
  for some ample line bundle \(L\) on \(X\),
  then there exists a proper closed subvariety \(Z\subsetneq X\) such that every nonconstant orbifold entire curve \(f\colon\C\to(X,\Delta)\) satisfies \(f(\C)\subseteq Z\).
\end{theo}

\begin{proof}
  Suppose there is a non trivial section \(s\in H^{0}\big(X,{\textstyle\bigoplus_{N\geq1}}S^{N}\Omega_{X,\Delta}\otimes L^{\vee}\big)\).
  Then by Corollary~\ref{coro:OFVT2}, any entire orbifold curve \(f\colon\C\to(X,\Delta)\) satisfies \(f^{\ast}s\equiv0\).
  In other words, \(f\) is tangent to the (multi-)foliation defined by \(s\).
  Then the same proof as in Theorem~\ref{theo:fol} implies that \(f\) is algebraically degenerate and \(f\) extends to a morphism \(f\colon(\P^{1}, \Delta') \to (X,\Delta)\), such that \(\deg(K_{(\P^{1}, \Delta')})\leq 0\).
  Theorem 6.6 in \cite{Rou10} gives that there are only finitely many such curves in \(X\), since \((X,\Delta)\) is of general type.
  This finite set defines a proper algebraic subset \(Z\subsetneq X\).
\end{proof}

As a consequence, one obtains the following orbifold version of results of Bogomolov and Mc Quillan \cite{McQ98} (see also \cite{Rou12}).
\begin{theo}
  \label{theo:order1}
  A smooth orbifold surface of general type \((X,\Delta)\) such that
  \[
    \chi_{1}(\pi,\Delta)
    =
    \big(c_{1}(\Omega_{\pi,\Delta})^{2}-c_{2}(\Omega_{\pi,\Delta})\big)
    >
    0
  \]
  satisfies the orbifold Green--Griffiths--Lang conjecture~\ref{conj:orbi-lang}.
\end{theo}
\begin{proof}
  By Riemann--Roch, if
  \(
  \chi_{1}(\pi,\Delta)
  =
  \big(c_{1}(\Omega_{\pi,\Delta})^{2}-c_{2}(\Omega_{\pi,\Delta})\big)
  >
  0
  \)
  then
  \[
    h^{0}(Y,S^{m}\Omega_{\pi,\Delta})+h^{2}(Y,S^{m}\Omega_{\pi,\Delta})
    \geq
    \frac{m^{3}}{6}(c_{1}(\Omega_{\pi,\Delta})^{2}-c_{2}(\Omega_{\pi,\Delta}))+O(m^{2}).
  \]
  Moreover by duality,
  \[
    H^{2}(Y,S^{m}\Omega_{\pi,\Delta})=H^{0}(Y,K_{Y}\otimes S^{m}\Omega_{\pi,\Delta} \otimes \mathcal{O}(-m \cdot \pi^{*}(K_{X}+\Delta)).
  \] 
  Since \(K_{X}+\Delta\) is big, for sufficiently large \(m\), \(m\cdot\pi^{\ast}(K_{X}+\Delta)-K_{Y}\) is effective.
  Then
  \[
    H^{2}(Y,S^{m}\Omega_{\pi,\Delta})
    \hookrightarrow
    H^{0}(Y,S^{m}\Omega_{\pi,\Delta}).
  \]
  This implies that
  \[
    h^{0}(Y,S^{m}\Omega_{\pi,\Delta})
    \geq
    \frac{m^{3}}{12}(c_{1}(\Omega_{\pi,\Delta})^{2}-c_{2}(\Omega_{\pi,\Delta}))+O(m^{2}).
  \]
  Therefore the orbifold cotangent bundle \(\Omega_{\pi,\Delta}\) is big and \((X,\Delta)\) satisfies the hypothesis of Theorem~\ref{theo:GGO1}.
\end{proof}

An interesting application of the preceding result is the following one, already discussed in the introduction.
\begin{coro}
  \label{coro:orbiCartan}
  Let \(X=\P^{2}\) and \(\Delta=\sum_{i=1}^{c} \left(1-\frac{1}{2}\right) L_{i}\) where \(L_{i}\) are lines in general position. If \(c\geq 11\) then \((X,\Delta)\) satisfies Conjecture~\ref{conj:orbi-lang}.
\end{coro}

More generally, we get:
\begin{coro}
  \label{coro:GGL-P2}
  Let \(\Delta\) be an orbifold divisor on \(\P^{2}\) with orbifold multiplicities \(m_{i}\geq2\).
  If \(\Delta\) has either
  \begin{itemize}
    \item
      at least \(4\) components of degree at least \(11\),
    \item
      at least \(5\) components of degree at least \(6\),
    \item
      at least \(6\) components of degree at least \(4\),
    \item
      at least \(7\) components of degree at least \(3\),
    \item
      at least \(8\) components of degree at least \(2\),
    \item
      or at least \(11\) components (of arbitrary degrees),
  \end{itemize}
  then \((\P^{2},\Delta)\) satisfies Conjecture~\ref{conj:orbi-lang}.
\end{coro}
\begin{proof}
  Considering the conjecture and the definition of orbifold curves,
  one can always remove some components (\ie{} take \(m_{i}=1\)),
  and one can always assume that all remaining orbifold multiplicities are equal to \(2\).
  Let us thus consider an orbifold divisor with \(c\) components, of respective degrees \(d_{1},\dotsc,d_{c}\), having all orbifold multiplicity \(2\).
  By Theorem~\ref{theo:order1}, it is then sufficient to prove that the orbifold pairs under consideration are of general type and satisfy \(\chi_{1}=s_{2}(\Omega_{\pi,\Delta})>0\).
  Namely, these have to satisfy
  \(d_{1}+\dotsb+d_{c}>6\)
  and
  \[
    \chi_{1}
    =
    \deg(\pi)
    \left(
      6
      -3\frac{\sum_{1\leq i\leq c}d_{i}}{2}
      +
      \frac{\sum_{1\leq i<j\leq c}d_{i}d_{j}-\sum_{1\leq i\leq c}d_{i}^{2}}{4}
    \right)
    >0.
  \]
  The first condition is clearly satisfied.
  The partial second derivative with respect to \(d_{i}\) of the second expression is (-1/2), whence it is a concave function.
  Let \(d_{m}\) be the minimum of the \(d_{i}\)'s and \(d_{M}\) their maximum.
  On the convex set \(\Set{d_{m}\leq d_{i}\leq d_{M}\forall i}\subseteq\R^{c}\), the minimum of the concave function under consideration is attained in an extremal point. At this point, \(c_{m}\) of the \(d_{i}\)'s have the value \(d_{m}\) and the others have the value \(d_{M}\).
  The minimum value is then
  \[
    6
    -
    3c_{m}
    \frac{d_{m}}{2}
    -
    3(c-c_{m})
    \frac{d_{M}}{2}
    +
    c_{m}(c_{m}-3)
    \frac{d_{m}^{2}}{8}
    +
    c_{m}(c-c_{m})
    \frac{d_{m}d_{M}}{4}
    +
    (c-c_{m})(c-c_{m}-3)
    \frac{d_{M}^{2}}{8}.
  \]
  Moreover, the derivative of this value with respect to \(d_{M}\) must be nonnegative, and the derivative with respect to \(d_{m}\) must be nonpositive, namely:
  \[
    (c-c_{m})
    \left(
      \frac{-3}{2}
      +
      c_{m}
      \frac{d_{m}}{4}
      +
      (c-c_{m}-3)
      \frac{d_{M}}{4}
    \right)
    \geq0
  \]
  and
  \[
    c_{m}
    \left(
      \frac{-3}{2}
      +
      (c_{m}-3)
      \frac{d_{m}}{4}
      +
      (c-c_{m})
      \frac{d_{M}}{4}
    \right)
    \leq0.
  \]
  One infers that if \(c_{m}\not\in\Set{0,c}\) then:
  \[
    \frac{3}{4}(d_{m}-d_{M})
    =
    \left(
      \frac{-3}{2}
      +
      c_{m}
      \frac{d_{m}}{4}
      +
      (c-c_{m}-3)
      \frac{d_{M}}{4}
    \right)
    -
    \left(
      \frac{-3}{2}
      +
      (c_{m}-3)
      \frac{d_{m}}{4}
      +
      (c-c_{m})
      \frac{d_{M}}{4}
    \right)
    \geq0.
  \]
  Therefore \(d_{m}=d_{M}\). Hence in any case, the minimum is attained in a point where all degrees are equal. We can thus assume that all degrees are \(d\). Then
  \[
    \chi_{1}
    =
    \deg(\pi)
    \left(
      6
      -\frac{3c}{2}d
      +\frac{c(c-3)}{8}d^{2}
    \right).
  \]
  It remains to check that this polynomial in \(d\) has a positive leading coefficients for \(c\geq4\), that its discriminant is negative for \(c>12\),
  and to compute the largest root
  for \(4\leq c\leq12\).
  These are easy computations.
\end{proof}

Up to passing to general hypersurfaces, we can strengthen the conclusion of Corollary~\ref{coro:GGL-P2} using Theorem~\ref{theo:hyperb}, since in all the considered cases \(\sum(1-\myfrac{1}{m_{i}})d_{i}>4\).
\begin{coro}
  \label{coro:GGL-P2+}
  If \(\Delta\) is a general orbifold divisor on \(\P^{2}\) satisfying the same assumptions,
  then all orbifold entire curves \(\C\to(\P^{2},\Delta)\) are constant.
\end{coro}
\begin{proof}
  By Corollary~\ref{coro:GGL-P2}, all orbifold entire curves \(\C\to(\P^{2},\Delta)\) are contained in algebraic curves. Therefore, by Theorem~\ref{theo:hyperb}, these are constant.
\end{proof}

\subsection{Existence of orbifold jet differentials on surfaces}
We will now consider higher order jet differentials.
We shall use the following vanishing theorem for orbifold tensors recently obtained by Guenancia and P\u{a}un.
\begin{theo}[\cite{GP16}]
  \label{theo:GP}
  Consider an adapted covering \(\pi\colon Y\to(X,\Delta)\) of a smooth orbifold pair with
  \(K_{X}+\Delta\) ample. For all \(r>s\) one has
  \[
    H^{0}\big(Y,\,
      (\T_{\pi,\Delta})^{\otimes r}
      \otimes
      (\Omega_{\pi,\Delta})^{\otimes s}
    \big)
    =
    \Set{0}.
  \]
\end{theo}
This result allows us to use the Riemann--Roch approach on surfaces.
\begin{coro}
  \label{coro:exist}
  Consider an adapted covering \(\pi\colon Y\to(X,\Delta)\) of a smooth orbifold surface of general type.
  For each integer \(k\) such that \(K_{X}+\Delta^{(k)}\) is ample:
  \[
    \dim H^{0}\big(Y,E_{k,N}\Omega_{\pi,\Delta}\big)
    \underset{N\gg1}{\geq}
    \chi(E_{k,N}\Omega_{\pi,\Delta}).
  \]
\end{coro}
\begin{proof}
  Since we are in the surface case, it is sufficient to prove that for large \(N\),
  \[
    H^{2}(Y, E_{k,N}\Omega_{\pi,\Delta})
    =
    \Set{0}.
  \]
  We use the graduation induced by the Green--Griffiths filtration
  \[
    \Grad E_{k,N}\Omega_{\pi,\Delta}
    =
    \bigoplus_{\Abs{\ell}=N}
    S^{\ell_{1}}\Omega_{\pi,\Delta^{(1)}}
    \otimes
    S^{\ell_{2}}\Omega_{\pi,\Delta^{(2)}}
    \otimes
    \dotsb
    \otimes
    S^{\ell_{k}}\Omega_{\pi,\Delta^{(k)}}.
  \]
  This shows that it is actually sufficient to prove that for all \(\ell\in\mathbb{N}^{k}\) with \(\Abs{\ell}=N\)
  \[
    H^{2}\big(Y,
      S^{\ell_{1}}\Omega_{\pi,\Delta^{(1)}}
      \otimes
      S^{\ell_{2}}\Omega_{\pi,\Delta^{(2)}}
      \otimes
      \dotsb
      \otimes
      S^{\ell_{k}}\Omega_{\pi,\Delta^{(k)}}
    \big)
    =
    \Set{0}.
  \]
  Using Serre duality, this is equivalent to
  \[
    H^{0}\big(Y,
      S^{\ell_{1}}\T_{\pi,\Delta^{(1)}}
      \otimes
      S^{\ell_{2}}\T_{\pi,\Delta^{(2)}}
      \otimes
      \dotsb
      \otimes
      S^{\ell_{k}}\T_{\pi,\Delta^{(k)}}
      \otimes
      \O(K_{Y})
    \big)
    =
    \Set{0}.
  \]
  Now, we remark that we have an injection
  \[
    S^{\ell_{1}}\T_{\pi,\Delta^{(1)}}
    \otimes
    S^{\ell_{2}}\T_{\pi,\Delta^{(2)}}
    \otimes
    \dotsb
    \otimes
    S^{\ell_{k}}\T_{\pi,\Delta^{(k)}}
    \hookrightarrow
    (\T_{\pi,\Delta^{(k)}})^{\otimes\abs{\ell}}.
  \]
  On the other hand, choosing \(p\) such that \(p\cdot\pi^{\ast}(K_{X}+\Delta^{(k)})-K_{Y}>0\), we obtain
  \[
    \O(K_{Y})
    \hookrightarrow
    \O(p\cdot\pi^{\ast}(K_{X}+\Delta^{(k)}))
    \hookrightarrow
    (\Omega_{\pi,\Delta^{(k)}})^{\otimes2p}.
  \]
  From Theorem~\ref{theo:GP}, we see that
  \[
    S^{\ell_{1}}\T_{\pi,\Delta^{(1)}}
    \otimes
    S^{\ell_{2}}\T_{\pi,\Delta^{(2)}}
    \otimes
    \dotsb
    \otimes
    S^{\ell_{k}}\T_{\pi,\Delta^{(k)}}
    \otimes
    \O(K_{Y})
    \hookrightarrow
    (\T_{\pi,\Delta^{(k)}})^{\otimes\abs{\ell}}
    \otimes
    (\Omega_{\pi,\Delta^{(k)}})^{\otimes2p}
  \]
  has no global sections as soon as \(\abs{\ell}>2p\).
  Since \(\abs{\ell}\geq\frac{N}{k}\), this is achieved as soon as \(N\) is large enough.
\end{proof}

\subsection{Projective plane}
We derive the following result on \(\P^{2}\), for smooth boundary divisors.
\begin{prop}
  \label{prop:order2}
  Every entire curve \(f\colon\C\to\P^{2}\) which ramifies over a smooth curve \(\mathcal{C}\) of degree \(d\geq12\) with sufficiently high order (\(\geq a_{\min}\) depending on \(d\)) satisfies an algebraic differential equation of order \(2\).
\end{prop}
\begin{table}[!ht]
  \[
    \begin{array}{||c|c||c|c||c|c||c|c||}
      \hline
      d&a_{\min}&
      d&a_{\min}&
      d&a_{\min}&
      d&a_{\min}
      \\\hline
      12&107&
      16&19&
      20\text{--}21&12&
      31\text{--}38&8
      \\
      13&44&
      17&16&
      22\text{--}23&11&
      39\text{--}60&7
      \\
      14&29&
      18&15&
      24\text{--}25&10&
      61\text{--}245&6
      \\
      15&22&
      19&13&
      26\text{--}30&9&
      246\text{--}\infty&5
      \\
      \hline
    \end{array}
  \]
  \caption{Minimal ramification orders for Prop.\ref{prop:order2}}
  \label{table:ramification.k=2}
\end{table}
\begin{proof}
  If \(a > 2d/(d-3)\) then \(K_{\P^{2}}+\Delta^{(2)}>0\), which allows us to apply Corollary~\ref{coro:exist}.
  Now, for \(k=2\), \(a\geq2\), Proposition~\ref{prop:RR} yields
  \[
    \chi\bigl(E_{2,N}\Omega_{\pi,\Delta}\bigr)
    =
    \frac{N^{5}}{1920}
    \frac{\deg(\pi)}{a^{2}}
    \Bigl(
      (48-27d+2d^{2})a^{2}
      -12(d-3)d a
      +12d^{2}
    \Bigr)
    +O\bigl(N^{4}\bigr).
  \]
  The result follows.
\end{proof}
\begin{rema}
  By Proposition~\ref{prop:nosections} below, jet order \(2\) is minimal for orbifold surfaces with smooth boundaries.
\end{rema}

\begin{rema}
  We have seen the asymptotic formula
  \[
    \chi\bigl(E_{k,N}\Omega_{\pi,\Delta}\bigr)
    =
    \frac{N^{(k+1)n-1}}{(k!)^{n}((k+1)n-1)!}
    \left(
      \frac{c_{1}(\Omega_{\pi,\Delta^{(\infty)}})^{n}}{n!}
      (\log k)^{n}
      +
      O\bigl((\log k)^{n-1}\bigr)
    \right)
    +O\bigl(N^{(k+1)n-2}\bigr).
  \]
  Since \(c_{1}^{2}(\P^{2})>0\), this Euler characteristic is always positive for \(k\) large enough.
  However, it is impossible to guarantee \(K_{X}+\Delta^{(k)}>0\) for such asymptotic jet orders \(k\).
\end{rema}

\subsection{Surfaces with trivial canonical bundle}
We shall now implement the Riemann--Roch approach in the interesting case of orbifold surfaces when the ambient surface has trivial canonical bundle.
\begin{theo}
  If \((X,\Delta)\) is a smooth orbifold surface with \(K_{X}\equiv0\), \(\Delta\) ample and \(\chi_{k}(\pi,\Delta)>0\), then for any ample line bundle \(L\to X\),
  \[
    H^{0}\big({\textstyle\bigoplus_{N\geq1}}E_{k,N}\Omega_{\pi,\Delta}\otimes L^{\vee}\big)
    \neq
    \Set{0}.
  \]
\end{theo}
\begin{proof}
  The case \(k=1\) follows at once from Proposition~\ref{prop:RR} and Corollary~\ref{coro:exist}.

  Assume now that \(\chi_{k}(\pi,\Delta)>0\) for \(k>1\).
  Since
  \[
  H^{0}(\oplus_{N\geq1}E_{k-1,N}\Omega_{\pi,\Delta}\otimes L^{\vee})
  \hookrightarrow
  H^{0}(\oplus_{N\geq1}E_{k,N}\Omega_{\pi,\Delta}\otimes L^{\vee}),
  \]
  reasoning by induction, one can moreover assume that \(\chi_{k-1}(\pi,\Delta)\leq0\).
  We then claim that \(K_{X}+\Delta^{(k)}>0\), and the result follows by Corollary~\ref{coro:exist}.

  Indeed, if not, then \(K_{X}+\Delta^{(k)}\equiv0\), \ie{} \(\Delta^{(k)}=\varnothing\) and
  \[
    \chi_{k}(\pi,\Delta)
    =
    \chi_{k-1}(\pi,\Delta)
    +
    \sum_{i=1}^{k-1}
    \frac{s_{1}(\Omega_{\pi,\Delta^{(i)}})}{i}
    \frac{s_{1}(\Omega_{\pi,\Delta^{(k)}})}{k}
    +
    \frac{s_{2}(\Omega_{\pi,\Delta^{(k)}})}{k^{2}}
    =
    \chi_{k-1}(\pi,\Delta)
    +
    \frac{s_{2}(\Omega_{X})}{k^{2}}.
  \]
  But by the classification of surfaces with trivial canonical bundle,
  \(s_{2}(\Omega_{X})=-c_{2}(X)\leq0\) and this yields a contradiction,
  since then \(0< \chi_{k}(\pi,\Delta) \leq \chi_{k-1}(\pi,\Delta) \leq 0\).
\end{proof}

\begin{coro}
  Let \((X,\Delta)\) be a smooth orbifold surface such that \(K_{X}\) is trivial and \(\abs{\Delta}\) is a smooth ample divisor.
  If the orbifold multiplicity is \(m\geq5\) and if \(c_{1}(\abs{\Delta})^{2}\geq 10c_{2}(X)\) then
  for any ample line bundle \(L\to X\),
  \[
    H^{0}\big({\textstyle\bigoplus}_{k,N\geq1}E_{k,N}\Omega_{\pi,\Delta}\otimes L^{\vee}\big)
    \neq
    \Set{0}.
  \]
\end{coro}
\begin{proof}
  Recall that for \(k\) big enough, the positivity of the Euler characteristic is given by the positivity of the coefficient
  \[
    \chi_{k}(\pi,\Delta)
    \bydef
    (-1)^{n}
    \!\!\sum_{q\in\mathbb{N}^{k}\colon\abs{q}=n}\!\!
    \frac{s_{q_{1}}(\Omega_{\pi,\Delta^{(1)}})}{1^{q_{1}}}
    \dotsm
    \frac{s_{q_{k}}(\Omega_{\pi,\Delta^{(k)}})}{k^{q_{k}}}.
  \]
  Now, from the residue short exact sequence:
  \[
    s(\Omega_{\pi,\Delta})
    =
    s(\Omega_{X})
    \prod_{i}
    \frac{(1-c_{1}(D_{i}))}{(1-c_{1}(D_{i})/m_{i})}.
  \]
  If \(X\) is a surface with trivial canonical bundle, a formal computation yields that for \(k\geq m_{i},\forall i\):
  \begin{multline*}
    \chi_{k}(\pi,\Delta)
    =
    -\sum_{1\leq j\leq k}
    \left(\frac{1}{j^{2}}\right)
    c_{2}(X)
    +\\
    \sum_{i_{1}<i_{2}}
    \left(
      \sum_{2\leq j_{1}\leq m_{i_{1}}}
      \frac{1}{j_{1}}
      \sum_{2\leq j_{2}\leq m_{i_{2}}}
      \frac{1}{j_{2}}
    \right)
    c_{1}(D_{i_{1}})c_{1}(D_{i_{2}})
    +\\
    \sum_{i}
    \left(
      \sum_{2\leq j_{1}<j_{2}\leq m_{i}}
      \frac{1}{j_{1}j_{2}}
      -\frac{(m_{i}-1)}{2m_{i}}
    \right)
    c_{1}(D_{i})^{2}.
  \end{multline*}
  Recall that \(c_{2}(X)\geq0\).
  In the one component case one gets:
  \[
    \chi_{k}(\pi,\Delta)
    \geq
    \left(
      \sum_{2\leq j_{1}<j_{2}\leq m}
      \frac{1}{j_{1}j_{2}}
      -\frac{(m-1)}{2m}
    \right)
    c_{1}(D)^{2}
    -
    \frac{\pi^{2}}{6}
    c_{2}(X).
  \]
  A numerical exploration shows that the coefficient \(c_{m}\) of \(c_{1}(D)^{2}\) becomes positive for \(m\geq5\) and that then \(\pi^{2}/(6c_{m})\leq10\).
\end{proof}
\begin{rema}
  The same proof shows that the result also holds \eg{} if \(\abs{\Delta}\) has several components \(\Delta_{i}\) with multiplicities \(m_{i}=2\) such that
  \[
    c_{1}(D)^{2}
    -
    \sum_{i}
    3c_{1}(D_{i})^{2}
    \geq
    \frac{4\pi^{2}}{3}
    c_{2}(X).
  \]
  (Anticipating the next section, notice that this of course never holds in the \(1\)-component case.)
\end{rema}

\section{Non-existence of orbifold jet differentials on varieties of general type}
\label{se:nonexistence}
The following results give evidence in support of Conjecture~\ref{conj:orbiDem}.
\subsection{Projective spaces}
We start with \(\P^{n}\), with a suitable smooth boundary divisor, giving examples of orbifolds of general type without any nonzero global jet differentials.
To see this, we first establish the following vanishing theorem for orbifold jet differentials, in the spirit of Diverio~\cite{Div08}.
\begin{prop}
  \label{prop:nosections}
  Take \(X=\P^{n}\) and \(\Delta=\left(1-\myfrac{1}{m}\right)H\), for a smooth hypersurface \(H\) of degree \(d\geq3\). If \(m\leq n\), then for any adapted covering \(\pi\colon Y\to(X,\Delta)\), for \(k\geq1\) and for \(N\geq1\), one has \(H^{0}(Y,E_{k,N}\Omega_{\pi,\Delta})=\Set{0}\).
  This vanishing holds without the assumption \(m\leq n\) when \(k<n\).
\end{prop}
\begin{proof}
  Suppose that for some \(k\) and \(N\), \(H^{0}(Y,E_{k,N}\Omega_{\pi,\Delta})\neq0\).
  Then one infers from Proposition~\ref{prop:filtration} that for some \(\ell_{1},\dotsc,\ell_{k}\) with \(\Abs{\ell}=N\)
  \[
    S^{\ell_{1}}\Omega_{\pi,\Delta^{(1)}}
    \otimes
    S^{\ell_{2}}\Omega_{\pi,\Delta^{(2)}}
    \otimes
    \dotsb
    \otimes
    S^{\ell_{k}}\Omega_{\pi,\Delta^{(k)}}
  \]
  has some nonzero global sections.
  Note that \(\Omega_{\pi,\Delta^{(\infty)}}=\pi^{\ast}\Omega_{\P^{n}}\).
  Since \(\T_{\P^{n}}\) is globally generated, one obtains nonzero global sections of
  \[
    S^{\ell_{1}}\Omega_{\pi,\Delta^{(1)}}
    \otimes
    \dotsb
    \otimes
    S^{\ell_{p}}\Omega_{\pi,\Delta^{(p)}}
  \]
  for the largest \(p\leq k\) such that \(p<m\) (\ie{} for which \(\Delta^{(p)}>\Delta^{(\infty)}=\varnothing\)).

  Remark that a nonzero section \(\sigma\) of \(E_{k,N}\Omega_{\pi,\Delta}\) can be made invariant to yield a nonzero section of \(E_{k,gN}\Omega_{X,\Delta}\), where \(g\) is the order of the Galois group of the covering \(\pi\colon Y\to(X,\Delta)\). It is obtained by taking the pushforward along \(\pi\) of the product of the Galois conjugates of \(\sigma\), which are all nonzero. Applying this result for \(k=1\), one deduces the existence of some nonzero global sections of
  \[
    S^{g\ell_{1}}\Omega_{X,\Delta^{(1)}}
    \otimes
    \dotsb
    \otimes
    S^{g\ell_{p}}\Omega_{X,\Delta^{(p)}}
    \subseteq
    S^{g\ell_{1}}\Omega_{\P^{n}}(\log H)
    \otimes
    \dotsb
    \otimes
    S^{g\ell_{p}}\Omega_{\P^{n}}(\log H).
  \]
  The bundle on the right is a product of symmetric powers of the logarithmic cotangent bundle of a hypersurface in \(\P^{n}\), with less than \(n\) factors.
  By the Pieri rule, all partitions in its direct sum decomposition into Schur powers have therefore less than \(n\) parts.
  This yields the sought contradiction, since these Schur powers have no sections, by the vanishing theorem of Brückmann--Rackwitz~\cite{BR90} (see~\cite{Div08,Div09}).
\end{proof}
\begin{exem}
  \label{exem:nosections}
  Take \(X=\P^{2}\) and \(\Delta=\left(1-\myfrac{1}{2}\right)\mathcal{C}\), where \(\mathcal{C}\) is a smooth curve of degree \(d\geq7\). It is a pair with ample canonical bundle such that \(H^{0}\big(Y,\bigoplus_{k,N\geq1}E_{k,N}\Omega_{\pi,\Delta}\big)=\Set{0}\), for any adapted covering \(\pi\colon Y\to(X,\Delta)\).
\end{exem}

\subsection{Abelian varieties}
\label{sse:abelian}
Let \(A\) be an Abelian variety of dimension \(n\geq2\) and let \(D\) be a smooth divisor on \(A\). We start again by proving a vanishing theorem for the logarithmic cotangent bundle.
\begin{prop}
  One has
  \[
    H^{0}(A,S^{\lambda}(\Omega_{A}(\log D))\otimes L^{\vee})
    \neq
    \Set{0}
  \]
  for an ample line bundle \(L\to A\)
  if and only if
  \(
  S^{\lambda}(\Omega_{A}(\log D))=(K_{A}(\log D))^{\otimes\lambda_{1}}
  \).
\end{prop}
\begin{proof}
  Let us first observe that \(\Omega_{A}(\log D)\) is nef.
  Since \(\Omega_{A}\) is globally generated, one is reduced to verify the nefness over \(D\).
  On \(D\), one has the following short exact sequence:
  \[0 \to \Omega_{D} \to \Omega_{A}(\log D)\rvert_{D} \to \mathcal{O}_{D} \to 0.\]
  Here, as a quotient of \(\Omega_{A}\rvert_{D}\), the vector bundle \(\Omega_{D}\) is nef.
  Thus, as an extension of nef vector bundles, \(\Omega_{A}(\log D)\rvert_{D}\) is nef.

  Consider a partition \(\lambda\), and recall (\eg{} \cite{Demailly88,Manivel}) that the Schur bundle \(S_{\lambda}(\Omega_{A}(\log D))\)
  is then the direct image of a nef line bundle \(\mathscr{L}\) on the flag bundle associated to \(\lambda\).
  Namely, let \(1\leq j_{1}<j_{2}<\dotsb<j_{m}\leq n\) be the jumps of \(\lambda\), for a certain \(m\leq n\) (\ie{} \(\lambda_{i}>\lambda_{i+1}\iff i\in\Set{j_{1},\dotsc,j_{m}}\)),
  and let \(F\) be the bundle of flags of subspaces with codimension \(j_{1},\dotsc,j_{m}\) in the fibers of \(\Omega_{A}(\log D)\).
  Let \(U^{j_{0}},\dotsc,U^{j_{m+1}}\) be the universal subbundles of codimension \(j_{0}<j_{1}<\dotsc<j_{m}\leq j_{m+1}\) on \(F\),
  where by convention
  \(j_{0}\bydef0\)
  and
  \(j_{m+1}\bydef n\).
  Then
  \[
    \mathscr{L}
    \bydef
    \bigotimes_{p=1}^{m}
    \det(U^{j_{p-1}}/U^{j_{p}})^{\otimes \lambda_{j_{p}}}.
  \]
  We will now study the bigness of \(\mathscr{L}\).
  To prove that \(\mathscr{L}\) is not big, it is sufficient to observe that the Segre number \(s_{n}(\mathscr{L})\) is zero.
  Using the Gysin formula from \cite[Prop.~1.2]{DP1} for the flag bundle \(F\to A\) (we transform a little bit), one gets the following expression for \(s_{n}(\mathscr{L})\):
  \[
    \resizebox{\linewidth}{!}{\(\displaystyle
      (-1)^{n}
      [t_{1}^{n}\dotsm t_{n}^{1}]
      \left(
        (\lambda_{1}t_{1}+\dotsb+\lambda_{n}t_{n})^{n}
        \tprod_{p=1}^{m}(t_{j_{p}+1}\dotsm t_{j_{p+1}})^{-j_{p}}
        \tprod_{1\leq i<j\leq n}(t_{i}-t_{j})
        \tprod_{1\leq i\leq n}t_{i}s_{\myfrac{1}{t_{i}}}(\Omega_{A}(\log D))
      \right),
    \)}
  \]
  where for a monomial \(m\) and a Laurent series \(P\) in the formal variables \(t_{1},\dotsc,t_{n}\), \([m](P)\) means the coefficient of \(m\) in \(P\).

  Now, the residue exact sequence on \(A\) reads as follows:
  \[0 \to \Omega_{A} \to \Omega_{A}(\log D) \to \mathcal{O}_{D} \to 0.\]
  Therefore, by the Whitney sum formula, we obtain the equality of total Segre classes:
  \[
    s(\Omega_{A}(\log D))=s(\Omega_{A})\cdot s(\mathcal{O}_{D})=s(\Omega_{A})\cdot c(\O(-D)).
  \]
  The last equality follows again from the Whitney sum formula applied on the short exact sequence
  \(0\to\O_{A}(-D)\to\O_{A}\to\O_{D}\to0\).
  The bundle \(\Omega_{A}\) being trivial we obtain \(s(\Omega_{A}(\log D))=1-c_{1}(D)\).
  Replacing in the above expression, the number \(s_{n}(\mathscr{L})\) becomes:
  \[
    (-1)^{n}
    [t_{1}^{n}\dotsm t_{n}^{1}]
    \left(
      (\lambda_{1}t_{1}+\dotsb+\lambda_{n}t_{n})^{n}
      \tprod_{p=1}^{m}(t_{j_{p}+1}\dotsm t_{j_{p+1}})^{-j_{p}}
      \tprod_{1\leq i<j\leq n}(t_{i}-t_{j})
      \tprod_{1\leq i\leq n}(t_{i}-c_{1}(D))
    \right).
  \]
  This coefficient is clearly a linear combination of \(1,\dotsc,c_{1}(D)^{n}\) but, for dimensional reasons, the only such number that is nonzero on \(A\) is \(c_{1}(D)^{n}\).
  In other words
  \[
    s_{n}(\mathscr{L})
    =
    [t_{1}^{n}\dotsm t_{n}^{1}]
    \left(
      (\lambda_{1}t_{1}+\dotsb+\lambda_{n}t_{n})^{n}
      \tprod_{p=1}^{m}(t_{j_{p}+1}\dotsm t_{j_{p+1}})^{-j_{p}}
      \tprod_{1\leq i<j\leq n}(t_{i}-t_{j})
    \right)
    c_{1}(D)^{n}.
  \]

  The degree of the polynomial under consideration is \( n(n+1)/2-\sum_{p=1}^{m}(j_{p+1}-j_{p})j_{p} \).
  As a consequence, if \(\sum_{p=1}^{m}(j_{p+1}-j_{p})j_{p}>0\), the coefficient of \(t_{1}^{n}\dotsm t_{n}^{1}\) is \(0\).
  To conclude, it remains to observe that \(\sum_{p=1}^{m}(j_{p+1}-j_{p})j_{p}=0\) if and only if \(S_{\lambda}(\Omega_{A}\log D)\) is a tensor power of the canonical bundle (\ie{} \(j_{1}=n=j_{m+1}\)).

  Now, since \(\mathscr{L}\) is relatively ample and \(p_{\ast}\mathscr{L}=S^{\lambda}\Omega_{A}(\log D)\), where \(p\colon F\to A\):
  \[
    \exists L\text{ ample, }
    H^{0}(A,S^{\lambda}\Omega_{A}(\log D)\otimes L^{\vee})\neq 0
    \implies
    \mathscr{L}\text{ big}.
  \]

  The only if direction follows directly from the fact that \((A,D)\) is of log general type.
\end{proof}
\begin{rema}
  Note that in general the bigness of \(\mathscr{L}\) is not equivalent to the bigness of the Serre line bundle on \(\P(S_{\lambda}(\Omega_{A}(\log D)))\).
  The first one is related to the sections of \(S^{m\lambda}\Omega_{A}(\log D)\) which is only a direct factor in \(S^{m}(S^{\lambda}\Omega_{A}(\log D))\)
  and also these line bundles could lie on bases with different dimensions.
  It is clear that if \(\lambda\) has \(n\) parts, \(S^{\lambda}(\Omega_{A}(\log D))\) is big.
  Indeed
  \[
    S^{\lambda}(\Omega_{A}(\log D))
    =
    (K_{A}(\log D))^{\otimes \lambda_{n}}
    \otimes
    S^{(\lambda_{1}-\lambda_{n},\dotsc,\lambda_{n-1}-\lambda_{n})}(\Omega_{A}(\log D)),
  \]
  which is the product of a big line bundle by a nef vector bundle.
\end{rema}

As an immediate corollary, we obtain examples of orbifolds of general type satisfying the Green--Griffiths--Lang Conjecture \ref{conj:orbi-lang} without any nonzero global jet differentials vanishing on an ample divisor.
\begin{coro}
  \label{coro:abelian.counterexample}
  Let \(A\) be an Abelian variety of dimension \(n\geq 2\) and \(D\) a smooth ample divisor on \(A\). Then, for any \(1<m\leq n\), the orbifold \((A,(1-\myfrac{1}{m})D)\) satisfies the Green--Griffiths--Lang Conjecture~\ref{conj:orbi-lang} but has no nonzero global jet differentials vanishing on an ample divisor.
\end{coro}
\begin{proof}
  By Theorem~\ref{theo:orbiab}, \((A,\Delta)\bydef(A, (1-\myfrac{1}{m})D)\) satisfies conjecture \ref{conj:orbi-lang}.

  Let \(\pi\colon Y\to(A,\Delta)\) be an adapted covering for the pair \((A,\Delta)\).
  Suppose that for some \(k\) and \(N\), \(H^{0}(Y,E_{k,N}\Omega_{\pi,\Delta}\otimes \pi^{\ast}L^{\vee})\neq0\) for some ample line bundle \(L\).
  Then, since \(m\leq n\), one infers from Proposition~\ref{prop:filtration} and from the triviality of \(\Omega_{\pi,\Delta^{(n)}}=\pi^{\ast}\Omega_{A}\) that for some \(\ell_{1},\dotsc,\ell_{n-1}\)
  \[
    S^{\ell_{1}}\Omega_{\pi,\Delta^{(1)}}
    \otimes
    \dotsb
    \otimes
    S^{\ell_{n-1}}\Omega_{\pi,\Delta^{(n-1)}}
    \otimes
    \pi^{\ast}L^{\vee}
  \]
  has some nonzero global sections.
  This would imply that
  \[
    S^{g\ell_{1}}\Omega_{A}(\log D)
    \otimes
    \dotsb
    \otimes
    S^{g\ell_{n-1}}\Omega_{A}(\log D)
    \otimes
    (L^{g^{n-1}})^{\vee}
  \]
  has nonzero global sections (\(g\bydef\abs{\Aut(\pi)})\)).
  Combined with the previous proposition, this yields a contradiction because according to the Pieri rule
  \(
  S^{g\ell_{1}}\Omega_{A}(\log D)
  \otimes
  \dotsb
  \otimes
  S^{g\ell_{n-1}}\Omega_{A}(\log D)
  \otimes
  (L^{g^{n-1}})^{\vee}
  \)
  is the direct sum of some Schur powers
  \(S^{\lambda}(\Omega_{1}(\log D))\otimes(L^{g^{n-1}})^{\vee}\)
  for partitions \(\lambda\) with at most \(n-1\) parts.
\end{proof}

\begin{rema}
  Let us recall that the key tool to obtain the degeneracy of orbifold entire curves  in Theorem~\ref{theo:orbiab} is Nevanlinna theory: more precisely, in \cite{YamAb} Yamanoi establishes a remarkable Second Main Theorem with the best truncation level one. Combined with an hypothesis of ramification (as in the definition of orbifold entire curves), one immediately gets the application to the orbifold setting.
  It is noteworthy that the proof of Yamanoi uses jet bundles and lifts of entire curves to jets spaces (but does not involve jet differentials vanishing on an ample divisor!). Earlier works by Siu and Yeung \cite{SY03} use meromorphic jet differentials to establish a Second Main Theorem with truncation level depending on the boundary divisor. In applications, especially in the orbifold setting, obtaining the lowest truncation level is very important.
\end{rema}

% K3 surfaces
\subsection{Kummer and ``general'' K3 surfaces}

We now show that the vanishing of orbifold jet differentials for Abelian surfaces gives a similar conclusion for Kummer \(K3\) surfaces and for ``general'' \(K3\) surfaces equipped with big and nef smooth divisors.

We first describe the situation and data relevant to the case of Kummer surfaces.

Let \(p_{0}\colon A_{0}\to S_{0}\) be the double cover from an Abelian surface \(A_{0}\) onto its associated Kummer quotient surface \(S_{0}\). Let \(D_{0}\subset S_{0}\) be a smooth irreducible ample divisor on \(S_{0}\) which avoids its \(16\) singular points. Let \(\alpha\colon A\to A_{0}\) (resp. \(\beta:S\to S_{0}\)) be the blow-up of the \(16\) corresponding points on \(A_{0}\) (resp. \(S_{0}\)), and \(p\colon A\to S\) the induced double cover. Let \(D\subset S\) the inverse image of \(D_{0}\) in \(S\), and \(D_{0}'\subset A_{0}, D'\subset A\) its inverse images there. Write \(\Delta_{0}\bydef(1-\frac{1}{2})D_{0}\), and similarly for its inverse images \(\Delta,\Delta'_{0},\Delta'\) on \(S,A_{0},A\).

Let \(\pi_{0}\colon Y_{0}\to S_{0}\) be a cover adapted to \((S_{0},\Delta_{0})\), so chosen that its ramification locus avoids the \(16\) singular points of \(S_{0}\). By base-changing with the relevant covers or blow-ups, we obtain covers \(\pi\colon Y\to S,\pi'_{0}\colon Y'_{0}\to A_{0},\pi'\colon Y'\to A'\) respectively adapted to \((S,\Delta)\), \((A_{0},\Delta'_{0})\) and \((A,\Delta')\). To simplify notation, we still denote with \(\beta\colon Y\to Y_{0}, p\colon Y'\to Y, \alpha\colon Y'\to Y'_{0}\) the maps induced by these base-changes.

For each \(k,N>0\), we thus also get natural injective sheaf maps: \(p^{*}\colon E_{k,N}\Omega_{\pi,\Delta}\to E_{k,N}\Omega_{\pi',\Delta'}\) and \(\alpha^{*}\colon E_{k,N}\Omega_{\pi'_{0},\Delta'_{0}}\to E_{k,N}\Omega_{\pi',\Delta'}\) which are isomorphic outside of the inverse images of the \(16\) singular points of \(S_{0}\).

We now denote by \(H_{0}\) a very ample line bundle on \(S_{0}\), and \(H,H',H'_{0}\) its inverse images on \(S,A, A_{0}\).

\begin{prop}
  \label{prop:kummer}
  The notations being as above, let \(B\) be an ample line bundle on \(S\), and \(B_{Y}\) its inverse image on \(Y\). Then: for any \(k,N>0\), \(H^{0}(Y,E_{k,N}\Omega_{\pi,\Delta}\otimes B_{Y}^{-1})=\Set{0}\).
\end{prop}
\begin{proof}
  The natural map:
  \[
    p^{\ast}
    \colon
    H^{0}(Y,E_{k,N}\Omega_{\pi,\Delta}\otimes B_{Y}^{-1})
    \to
    H^{0}(Y',E_{k,N}\Omega_{\pi',\Delta'}\otimes p^{*}(B_{Y}^{-1}))
  \]
  is obviously injective, and by Hartogs theorem, the natural map:
  \[
    \alpha^{\ast}
    \colon
    H^{0}(Y'_{0},E_{k,N}\Omega_{\pi'_{0},\Delta'_{0}}\otimes {B'}_{0}^{-1})
    \to
    H^{0}(Y',E_{k,N}\Omega_{\pi',\Delta'}\otimes \alpha^{*}({B'}_{0}^{-1}))
  \]
  is isomorphic, for any ample line bundle \(B'_{0}\) on \(A_{0}\) (its inverse image on \(Y'_{0}\) being written in the same way). From the (proof of the) preceding Corollary \ref{coro:abelian.counterexample}, we know that \(H^{0}(Y'_{0},E_{k,N}\Omega_{\pi'_{0},\Delta'_{0}}\otimes {B'}_{0}^{-1})=\{0\}\). This implies the claimed vanishing, since \(k.\alpha^{*}({B'}_{0})-p^{*}(B)\) is effective, for \(k\) big enough.
\end{proof}

We now consider the preceding orbifold pair \((S,\Delta)\), together with a marking for \(H^{2}(S,\Z)\). Notice that the class \([D]\) of \(D=2.\Delta\) in \(H^{2}(S,\Z)\) is what is called a ``pseudo-polarisation'' (\ie{} a big and nef class) in \cite{SP}. Associated to the pair \((S,[D])\) is a (nonseparated) fine moduli space of marked projective \(K3\) surfaces \((S_{t},[D_{t}]), t\in T\). Let \(f\colon\Sigma\to T\) be the associated family of \(K3\) surfaces, together with the line bundle \(\mathcal D'\) on \(\Sigma'\) inducing \(D_{t}\) on \(S_{t}\), for each \(t\in T'\) (base-changing from \(T\) to \(\P(f_{*}(\mathcal D))\bydef T')\), indicated with a ``prime'' subscript. The map \(f'\) being locally projective, we can (locally) construct a simultaneous cover \(\pi'\colon \mathcal Y'\to \Sigma'\) adapted to \(\mathcal D'\). We now consider the direct image sheaves
\(
\mathcal E_{k,N}
\bydef
(f'\circ \pi')^{\ast}(E_{k,N}\Omega_{\pi',\Delta'_{T}}
\otimes
\mathcal B^{-1})
\),
for \(\mathcal B\) relatively ample on \(\mathcal Y'\), and \(\Delta'_{T}\bydef\frac{1}{2}\cdot\mathcal{D}'\). By the preceding Proposition~\ref{prop:kummer}, these sheaves all vanish for \(t=0\), with \((S,\Delta)_{t=0}\) our initial Kummer orbifold pair. We thus deduce that these sheaves all vanish for \(t\) ``general'' in \(T'\) (that is: outside of a countable union of proper Zariski closed subsets of \(T'\)).

\begin{rema}
  One can of course wonder whether this result holds for all pairs \((S,\frac{1}{2}\cdot D)\) with \(S\) an arbitrary \(K3\) surface and \(D\) an ample smooth divisor, or even for \((X,\frac{1}{m}\cdot D)\) for \(X\) projective with \(K_{X}\) trivial, \(D\) smooth ample, and \(m\leq n\bydef\dim(X)\).
  For the ``general'' member of the known families of Hyperkähler manifolds, the preceding argument can probably be adapted, but it would be more interesting to have an intrinsic, deformation-free, argument.
\end{rema}

\begin{rema}
  Example~\ref{exem:nosections}, Corollary~\ref{coro:abelian.counterexample} and Proposition~\ref{prop:kummer}
  show clearly that in the general orbifold situation,
  one cannot expect to fully establish the Green--Griffiths--Lang conjecture
  by using only the approach of jet bundles.
  Corollary~\ref{coro:abelian.counterexample} proves the left-to-right direction of Conjecture~\ref{conj:orbiDem} for Abelian varieties.
\end{rema}
\begin{rema}
  Corollary~\ref{coro:orbiCartan} and Corollary~\ref{coro:abelian.counterexample} also illustrate that Nevanlinna theory and the theory of orbifold jet differentials introduced in this paper produce positive complementary results towards the orbifold Green--Griffiths--Lang conjecture.
\end{rema}

\begin{acknowledgements}
  E.R. would like to thank Mihai P{\u{a}}un and Nessim Sibony for fruitful discussions related to this article.
  The authors would like to thank Jean-Pierre Demailly for pointing out that we were not considering enough morphisms in the case of rational orbifold multiplicities. This remark has led to Definition~\ref{defi:corr}.
\end{acknowledgements}

\bibliographystyle{amsalpha}
\bibliography{CM6404}
\end{document}